\documentclass[twoside,11pt,leqno]{article}        
\usepackage{amssymb, amsmath, amscd, amsthm, color, epsfig, url, graphicx, xr, makeidx}  
\usepackage{amsthm}
\usepackage{psfrag}
\usepackage{enumerate}
\usepackage{color}
\usepackage[all]{xy}

\newtheorem{thm}{Theorem}[section]

\newtheorem{cor}[thm]{Corollary}

\newtheorem{conj}[thm]{Conjecture}

\newtheorem{thm1}{Theorem}[subsection]

\newtheorem{cor1}[thm1]{Corollary}
\newtheorem{prop1}[thm1]{Proposition}

\theoremstyle{definition}                  
\newtheorem{rem1}[thm1]{Remark} 
     
\newtheorem{exam1}[thm1]{Example}
\newtheorem{defi1}[thm1]{Definition}
\newtheorem{defi}[thm]{Definition}
\newtheorem{rem}[thm]{Remark} 

\newcommand{\refS}[1]{Section~\ref{S:#1}}
\newcommand{\refT}[1]{Theorem~\ref{T:#1}}

\newcommand{\refP}[1]{Proposition~\ref{P:#1}}
\newcommand{\refD}[1]{Definition~\ref{D:#1}}

\newcommand{\refEx}[1]{Example~\ref{Ex:#1}}

\newcommand{\R}{{\mathbb R}}

\newcommand{\Q}{{\mathbb Q}}

\newcommand{\hofiber}{\operatorname{hofiber}}

\newcommand{\holim}{\operatorname{holim}}

\newcommand{\Map}{\operatorname{Map}}
\newcommand{\Emb}{\operatorname{Emb}}

\newcommand{\Imm}{\operatorname{Imm}}

\newcommand{\Hom}{\operatorname{Hom}}

\newcommand{\Conf}{\operatorname{Conf}}
\newcommand{\Ho}{\operatorname{H}}

\newcommand{\calV}{{\mathcal{V}}}

\newcommand{\calC}{{\mathcal{C}}}
\newcommand{\calD}{{\mathcal{D}}}

\newcommand{\K}{{\mathcal{K}}}

\newcommand{\Lk}{{\mathcal{L}}}
\newcommand{\HLk}{{\mathcal{H}}}
\newcommand{\Br}{{\mathcal{B}}}

\newcommand{\LD}{{\mathcal{LD}}}
\newcommand{\HLD}{{\mathcal{HD}}}

\newcommand{\LW}{{\mathcal{LW}}}
\newcommand{\HLW}{{\mathcal{HW}}}

\newcommand{\Top}{\operatorname{Top}}

\newcommand{\rnf}{\renewcommand{\thefootnote}{\arabic{footnote}}}
\newcommand{\thankyou}[2]{\stepcounter{footnote}\footnotetext[#1]{#2}}

\title{\vspace{-3cm}
\ \hspace{-2.70in}
{\small {\bf Morfismos},
Vol. XX, No. XX, 201X,                                                                                                
pp.
        X--X   
}
\\
\vspace{3cm} 
Configuration space integrals and the topology of knot and link spaces 
}

\author{\rnf   Ismar Voli\'c
               \footnotemark[1]
}

\date{}

\begin{document}
\maketitle


\thankyou{1}{The author was supported by the National Science Foundation grant DMS 1205786.}




\begin{abstract} \noindent
This article surveys the use of configuration space integrals in the study of the topology of knot and link spaces.  The main focus is the exposition of how these integrals produce finite type invariants of classical knots and links.  More generally,  we also explain the construction of a chain map, given by configuration space integrals, between a certain diagram complex and the deRham complex of the space of knots in dimension four or more.  A generalization to spaces of links, homotopy links, and braids is also treated, as are connections to Milnor invariants, manifold calculus of functors, and the rational formality of the little balls operads.
\end{abstract}

\noindent 


\vspace{.7cm}\centerline{\sc Contents}
1 ~\hspace{.6mm}Introduction\hfill\pageref{S:Intro}\ \ \ \ 

~~~~1.1 ~\hspace{.6mm}Organization of the paper\hfill\pageref{S:Organization}\ \ \ \ 

2 ~\hspace{.6mm}Preliminaries\hfill\pageref{S:Preliminaries}\ \ \ \ 

~~~~2.1 ~\hspace{.6mm}Differential forms and integration along the fiber\hfill\pageref{S:Forms}\ \ \ \ 

~~~~2.2 ~\hspace{.6mm}Space of long knots\hfill\pageref{S:KnotsLinks}\ \ \ \ 

~~~~2.3~\hspace{.6mm}Finite type invariants\hfill\pageref{S:FiniteType}\ \ \ \ 

~~~~2.4~\hspace{.6mm}Configuration spaces and their compactification\hfill\pageref{S:Compactification}\ \ \ \ 

3 ~\hspace{.6mm}Configuration space integrals and finite type knot \hfill \ \ \ \ \ 

~~~~invariants\hfill\pageref{S:ConfFTInvs}\ \ \ \ 

~~~~3.1 ~\hspace{.6mm}Motivation: The linking number\hfill\pageref{S:LinkingNumber}\ \ \ \ 

~~~~3.2 ~\hspace{.6mm}â``Self-linking" for knots\hfill\pageref{S:Self-Linking}\ \ \ \ 

~~~~3.3 ~\hspace{.6mm}Finite type two knot invariant\hfill\pageref{S:Casson}\ \ \ \ 

~~~~3.4~\hspace{.6mm}Finite type $k$ knot invariants\hfill\pageref{S:GeneralFiniteType}\ \ \ \ 

4 ~\hspace{.6mm}Generalization to $\K^n$, $n>3$\hfill\pageref{S:HigherKnots}\ \ \ \ 

5 ~\hspace{.6mm}Further generalizations and applications\hfill\pageref{S:Generalizations}\ \ \ \ 

~~~~5.1~\hspace{.6mm}Spaces of links\hfill\pageref{S:Links}\ \ \ \ 

~~~~5.2~\hspace{.6mm}Manifold calculus of functors and finite type \hfill \ \ \ \ \ 

 ~~~~invariants\hfill\pageref{S:Calculus}\ \ \ \ 

~~~~5.3~\hspace{.6mm}Formality of the little balls operad\hfill\pageref{S:Formality}\ \ \ \ 

\ \hspace{-2.5mm}References\hfill\pageref{references}\ \ \ \ 

\vspace{.22cm}


\section{Introduction}\label{S:Intro}


Configuration space integrals are fascinating objects that lie at the intersection of physics, combinatorics, topology, and geometry.  Since their inception over twenty years ago, they have emerged as an important tool in the study of the topology of spaces of embeddings and in particular of spaces of knots and links.
\bigskip

The beginnings of configuration space integrals can be traced back to Guadagnini, Martellini, and Mintchev \cite{GMM} and Bar-Natan \cite{BN:Thesis} whose work was inspired by Chern-Simons theory.  The more topological point of view was introduced by Bott and Taubes \cite{BT}; configuration space integrals are because of this sometimes even called \emph{Bott-Taubes integrals} in the literature (more on Bott and Taubes' work can be found in \refS{Casson}).  The point of this early work was to use configuration space integrals to construct a knot invariant in the spirit of the classical linking number of a two-component link.  
This invariant turned out to be of \emph{finite type} (finite type invariants are reviewed in \refS{FiniteType}) and D. Thurston \cite{Thurs} generalized it to construct all finite type invariants. We will explain D. Thurston's result in \refS{GeneralFiniteType}, but the idea is as follows:

Given a trivalent diagram $\Gamma$ (see \refS{FiniteType}), one can construct a bundle 
$$
\pi\colon\Conf[p,q;\K^3, \R^n]\longrightarrow\K^3,
$$
where $\K^3$ is the space of knots in $\R^3$.  Here $p$ and $q$ are the numbers of certain kinds of vertices in $\Gamma$ and $ \Conf[p,q;\K^3, \R^n]$ is a pullback space constructed from an evaluation map and a projection map.  The fiber of $\pi$ over a knot $K\in\K^3$ is the compactified configuration space of $p+q$ points in $\R^3$, first $p$ of which are constrained to lie on $K$.  The edges of $\Gamma$ also give a prescription for pulling back a product of volume forms on $S^{2}$ to $\Conf[p,q;\K^3, \R^n]$.  The resulting form can then be integrated along the fiber, or pushed forward, to $\K^3$.  The dimensions work out so that this is a 0-form and, after adding the pushforwards over all trivalent diagrams of a certain type, this form is in fact closed, i.e.~it is an invariant.  Thurston then proves that this is a finite type invariant and that this procedure gives all finite type invariants.
\bigskip

The next generalization was carried out by Cattaneo, Cotta-Ra\-mu\-si\-no, and Longoni \cite{CCRL}.  Namely, let $\K^n$, $n>3$, be the space of knots in $\R^n$.  The main result of \cite{CCRL} is that there is a cochain map
\begin{equation}\label{E:CCRLMap}
\mathcal{D}^n\longrightarrow \Omega^*(\K^n)
\end{equation}
between a certain diagram complex $\mathcal{D}^n$ generalizing trivalent diagrams and the deRham complex of $\K^n$.  The map is given by exactly the same integration procedure as Thurston's, except the degree of the form that is produced on $\K^n$ is no longer zero.  Specializing to classical knots (where there is no longer a cochain map due to the so-called ``anomalous face"; see \refS{GeneralFiniteType}) and degree zero, one recovers the work of Thurston.  Cattaneo, Cotta-Ramusino, and Longoni have used the map \eqref{E:CCRLMap} to show that spaces of knots have cohomology in arbitrarily high degrees in \cite{CCRL:Struct} by studying certain algebraic structures on $\mathcal{D}^n$ that correspond to those in the cohomology ring of $\K^n$.  Longoni also proved in \cite{Long:Classes} that some of these classes arise from non-trivalent diagrams. 
\bigskip

Even though configuration space integrals were in all of the aforementioned work constructed for ordinary closed knots, it has in recent years become clear that the variant for long knots is also useful. Because some of the applications we describe here have a slight  preference for the long version, this  is the space we will work with.  The difference between the closed and the long version is minimal from the perspective of this paper, as  explained at the beginning of \refS{KnotsLinks}.
\bigskip

More recently, configuration space integrals have been generalized to (long) links, homotopy links, and braids \cite{KMV:FTHoLinks, V:B-TLinks}, and this work is summarized in \refS{Links}.  One nice feature of this generalization is that it provides the connection to Milnor invariants.  This is because configuration space integrals give finite type invariants of homotopy links, and, since Milnor invariants are finite type, this immediately gives integral expressions for these classical invariants.
\bigskip

We also describe two more surprising applications of configuration space integrals.  Namely, one can use \emph{manifold calculus of functors} to place finite type invariants in a more homotopy-theoretic setting as described in \refS{Calculus}.  Functor calculus also combines with the \emph{formality of the little $n$-discs operad}  to give a description of the rational homology of $\K^n$, $n>3$.  Configuration space integrals play a central role here since they are at the heart of the proof of operad formality.  Some details about this are provided in \refS{Formality}.
\bigskip

In order to keep the focus of this paper on knot and links and keep its length to a manageable size, we will regrettably  only point the reader to three other topics that are growing in promise and popularity.  The first is the work of Sakai \cite{Sakai:BTLongPlanes} and its expansion by Sakai and Watanabe \cite{SW:1-loopGraph} on \emph{long planes}, namely embeddings of $\R^k$ in $\R^n$ fixed outside a compact set.  These authors use configuration space integrals to produce nontrivial cohomology classes of this space with certain conditions on $k$ and $n$. This work generalizes classes produced by others \cite{CR:Wilson, Wat:n-knots} and complements recent work by Arone and Turchin \cite{AT:LongPlanes2} who show, using homotopy-theoretic methods, that the homology of $\Emb(\R^k,\R^n)$ is given by a certain graph complex for $n\geq 2k+2$.  Sakai has further used configuration space integrals to produce a cohomology class of $\K^3$ in degree one that is related to the Casson invariant  \cite{Sakai:Nontrivalent} and has given a new interpretation of the Haefliger invariant for $\Emb(\R^k,\R^n)$ for some $k$ and $n$ \cite{Sakai:BTLongPlanes}.  In an interesting bridge between two different points of view on spaces of knots, Sakai has in \cite{Sakai:BTLongPlanes} also combined the configuration space integrals with Budney's action of the little discs operad on $\K^n$ \cite{Bud:LCLK}. 
\bigskip

The other interesting development is the recent work of Koytcheff \cite{Koyt:HomotopyBT} who develops a homotopy-theoretic replacement of configuration space integrals.  He uses the Pontryagin-Thom construction to ``push forward" forms from $\Conf[p,q;\K^n,\R^n]$ to $\K^n$.  The advantage of this approach is that is works over any coefficients, unlike ordinary configuration space integration, which takes values in $\R$.  A better understanding of how Koytcheff's construction relates to the original configuration space integrals is still needed.
\bigskip

The third topic is the role configuration space integrals have recently played in the construction of asymptotic finite type invariants of divergence-free vector fields \cite{KV:FinTypeVectFields}.  The approach in this work is to apply configuration space integrals to trajectories of a vector field.  In this way, generalizations of some familiar asymptotic vector field invariants like asymptotic linking number, helicity, and the asymptotic signature can be derived.
\bigskip

Lastly, some notes on the style and expositional choices we have made in this paper are in order.  We will assume an informal tone, especially at times when writing down something precisely would require us to introduce cumbersome notation.   To quote from a friend and coauthor Brian Munson \cite{M:MfldCalc}, ``we will frequently omit arguments which would distract us from our attempts at being lighthearted".
Whenever this is the case, a reference to the place where the details appear will be supplied. In particular, most of the proofs we present here have been worked out in detail elsewhere, and if we feel that the original source is sufficient, we will simply give a sketch of the proof and provide ample references for further reading.  It is also worth pointing out that many open problems are stated througout and our ultimate hope is that, upon looking at this paper, the reader will be motivated to tackle some of them.


\subsection{Organization of the paper}\label{S:Organization}


We begin by recall some of the necessary background in \refS{Preliminaries}.  We only give the basics but furnish abundant references for further reading.  In particular, we review integration along the fiber in \refS{Forms} and pay special attention to integration for infinite-dimensional manifolds and manifolds with corners.  In \refS{KnotsLinks} we define the space of long knots and state some observations about it.  A review of finite type invariants is provided in \refS{FiniteType}; they will play a central role later.  This section also includes a discussion of chord diagrams and trivalent diagrams. Finally in \refS{Compactification}, we talk about configuration spaces and their Fulton-MacPherson compactification.  These are the spaces over which our integration will take place.
\bigskip

\refS{ConfFTInvs} is devoted to the construction of finite type invariants via configuration space integrals.  The motivating notion of the linking number is recalled in \refS{LinkingNumber}, and that leads to the failed construction of the ``self-linking" number in \refS{Self-Linking} and its improvement to the simplest finite type (Casson) invariant in \refS{Casson}.  This section is at the heart of the paper since it gives all the necessary ideas for all of the constructions encountered from then on.  Finally in \refS{GeneralFiniteType} we construct all finite type knot invariants via configuration space integrals.
\bigskip

\refS{HigherKnots} is dedicated to the description of the cochain map \eqref{E:CCRLMap} and includes the definition of the cochain complex $\mathcal D^n$.  We also discuss how this generalizes D. Thurston's construction that yields finite type invariants.
\bigskip

Finally in \refS{Generalizations}, we give brief accounts of some other features, generalizations, and applications of configuration space integrals.  More precisely, in \refS{Links}, we generalize the constructions we will have seen for knots to links, homotopy links, and braids; in \refS{Calculus}, we explore the connections between manifold calculus of functors and configuration space integrals; and in \refS{Formality}, we explain how configuration space integrals are used in the proof of the formality of the little $n$-discs operad and how this leads to information about the homology of spaces of knots.


\section{Preliminaries}\label{S:Preliminaries}



\subsection{Differential forms and integration along the fiber}\label{S:Forms}


The strategy we will employ in this paper is to produce differential forms on spaces of knots and links via configuration space integrals.
Since introductory literature on differential forms is abundant (see for example \cite{BT:DiffForms}), we will not recall their definition here.    We also assume the reader is familiar with integration of forms over manifolds.
\bigskip

We will, however, recall some terminology that will be used throughout:  Given a smooth oriented manifold $M$, one has the \emph{deRham cochain complex $\Omega^*(M)$ of differential forms}:
$$
0\longrightarrow \Omega^0(M)\stackrel{d}{\longrightarrow} \Omega^1(M)\stackrel{d}{\longrightarrow} \Omega^2(M)\longrightarrow\cdots
$$
where $\Omega^k(M)$ is the space of smooth $k$-forms on $M$.  The differential $d$ is the \emph{exterior derivative}.  A form $\alpha\in\Omega^k(M)$ is \emph{closed} if $d\alpha=0$ and \emph{exact} if $\alpha=d\beta$ for some $\beta\in\Omega^{k-1}(M)$.  The $k$th deRham cohomology group of $M$, $\Ho^k(M)$, is defined the usual way as the kernel of $d$ modulo the image of $d$, i.e.~the space of closed forms modulo the subspace of exact forms. All the cohomology we consider will be over $\R$. 
\bigskip

The \emph{wedge product}, or \emph{exterior product}, of differential forms gives $\Omega^*(M)$ the structure of an algebra, called the \emph{deRham}, or \emph{exterior algebra of $M$}.
\bigskip

According to the deRham Theorem, deRham cohomology is isomorphic to the ordinary singular cohomology.  In particular, $\Ho^0(M)$ is the space of functionals on connected components of $M$, i.e.~the space of \emph{invariants} of $M$.  The bulk of this paper is concerned with invariants of knots and links.
\bigskip

The complex $\Omega^*(M)$ can be defined for manifolds with boundary by simply restricting the form to the boundary.  Locally, we take restrictions of forms on open subsets of $\R^k$ to $\R^{k-1}\times \R_+$.  Further, one can define differential forms on \emph{manifolds with corners} (an $n$-dimensional manifold with corners is locally modeled on $\R^k_+\times \R^{n-k}$, $0\leq k\leq n$; see \cite{Joyce:MfldsCrnrs} for a nice introduction to these spaces) in exactly the same fashion by restricting forms to the boundary components, or \emph{strata} of $M$.
\bigskip

The complex $\Omega^*(M)$ can also be defined for infinite-dimensional manifolds such as the spaces of knots and links we will consider here.  One usually considers the forms on the vector space on which $M$ is locally modeled and then patches them together into forms on all of $M$.  When $M$ satisfies conditions such as paracompactness, this ``patched-together" complex again computes the ordinary cohomology of $M$.  Another way to think about forms on an infinite-dimensional manifold $M$ is via the \emph{diffeological} point of view which considers forms on open sets mapping into $M$.  For more details, see \cite[Section 2.2]{KMV:FTHoLinks} which gives further references.
\bigskip

Given a smooth fiber bundle $\pi\colon E\to B$ whose fibers are compact oriented $k$-dimensional  manifolds, there is a map
\begin{equation}\label{E:Pushforward}
\pi_*\colon \Omega^n(E)\longrightarrow \Omega^{n-k}(B)
\end{equation}
called the \emph{pushforward} or \emph{integration along the fiber}.  The idea is to define the form on $B$ pointwise by integrating over each fiber of $\pi$.  Namely, since $\pi$ is a bundle, each point $b\in B$ has a $k$-dimensional neighborhood $U_b$ such that $\pi^{-1}(b)\cong B\times U_b$.  Then a form $\alpha\in \Omega^n(E)$ can be restricted to this fiber, and the coordinates on $U_b$ can be ``integrated away".  The result is a form on $B$ whose dimension is that of the original form but reduced by the dimension of the fiber.  The idea is to then patch these values together into a form on $B$.  Thus the map \eqref{E:Pushforward} can be described by
\begin{equation}\label{E:Pushforward2}
\alpha\longmapsto
\left(
b\longmapsto \int_{\pi^{-1}(b)}\alpha
\right).
\end{equation}

In terms of evaluation on cochains, $\pi_*\alpha$ can be thought of as an $(n-k)$-form on $B$ who value on a $k$-chain $X$ is 
$$
\int_X \pi_*\alpha=\int_{\pi^{-1}(X)}\alpha.
$$ 
For introductions to the pushforward, see \cite{BT:DiffForms, HG:Pushforward}.

\begin{rem1}
The assumption that the fibers of $\pi$ be compact can be dropped, but then forms with compact support should be used to guarantee the convergence of the integral.
\end{rem1}

To check if $\pi_*\alpha$ is a closed form, it suffices to integrate $d\alpha$, i.e.~we have \cite[Proposition 6.14.1]{BT:DiffForms}
$$
d\pi_*\alpha=\pi_*d\alpha.
$$
The situation changes when the fiber is a manifold with boundary or with corners, as will be the case for us.  Then by Stokes' Theorem, $d\pi_*\alpha$ has another term.  Namely, we have
\begin{equation}\label{E:dStokes}
d\pi_*\alpha=\pi_*d\alpha + (\partial\pi)_*\alpha
\end{equation}
where $(\partial\pi)_*\alpha$ is the integral of the restriction of $\alpha$ to the codimension one boundary of the fiber.  This can be seen by an argument similar to the proof of the ordinary Stokes' Theorem.  For Stokes' Theorem for manifolds with corners, see, for example \cite[Chapter 10]{Lee:SmthMflds}.
\bigskip

In the situations we will encounter here, $\alpha$ will be a closed form, in which case $d\alpha=0$.  Thus the first term in the above formula vanishes, so that we get
\begin{equation}\label{E:dStokesClosed}
d\pi_*\alpha=(\partial\pi)_*\alpha.
\end{equation}
Our setup will combine various situations described above -- we will have a   smooth bundle $\pi\colon E\to B$ of infinite-dimensional spaces with fibers that are finite-dimensional compact manifolds with corners.  


\subsection{Space of long knots}\label{S:KnotsLinks}


We will be working with \emph{long} knots and links (links will be defined in \refS{Links}), which are easier to work with than ordinary closed knots and links in many situations.  For example, long knots are $H$-spaces via the operation of stacking, or concatenation, which gives their (co)homology groups more structure.  Also, the applications of manifold calculus of functors to these spaces (\refS{Calculus}) prefer the long model.  Working with long knots is not much different than working with ordinary closed ones since the theory of long knots in $\R^n$ is the same as the theory of based knots in $S^n$.  From the point of view of configuration space integrals, the only difference is that, for the long version, we will have to consider certain \emph{faces at infinity} (see Remark \ref{R:FacesAtInfinity}).
\bigskip

Before we define long knots, we remind the reader that a smooth map $f\colon M\to N$ between smooth manifolds $M$ and $N$ is
\begin{itemize}
\item an \emph{immersion} if the derivative of $f$ is everywhere injective;
\item an \emph{embedding} if it is an immersion and a homeomorphism onto its image.
\end{itemize}

Now let $\Map_c(\R, \R^n)$ be the space of smooth maps of $\R$ to $\R^n$, $n\geq 3$, which outside the standard interval $I$ (or any compact set, really) agrees with the map 
\begin{align*}
\R & \longrightarrow \R^n \\
t & \longmapsto (t,0,0,...,0).
\end{align*}
\bigskip

Give $\Map_c(\R, \R^n)$ the $\calC^{\infty}$ topology (see, for example, \cite[Chapter 2]{Hirsch:DiffTop}).  
\bigskip

\noindent\begin{defi1}\label{D:KnotsAndLinks_}
Define the space of 
\emph{long (or string) knots} $\K^n\subset \Map_c(\R, \R^n)$ as the subset of maps $K\in \Map_c(\R, \R^n)$ that are embeddings, endowed with the subspace topology.
\end{defi1}
A related space that we will occassionally have use for is the space of \emph{long immersions} $\Imm_c(\R, \R^n)$ defined the same way as the space of long knots except its points are immersions.   Note that $\K^n$ is a subspace of $\Imm_c(\R, \R^n)$.
\smallskip

A homotopy in $\K^n$ (or any other space of embeddings) is called an \emph{isotopy}.  A homotopy in $\Imm_c(\R, \R^n)$ is a \emph{regular homotopy}.
\bigskip

An example of a long knot is given in Figure \ref{F:KnotExample}. Note that we have confused the map $K$ with its image in $\R^n$.  We will do this routinely and it will not cause any issues.
\begin{figure}[!htbp]
\begin{center}
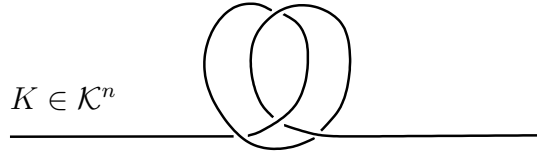
\caption{An example of a knot in $\R^n$. }
\label{F:KnotExample}
\end{center}
\end{figure} 

From now on, the adjective ``long" will be dropped; should we need to talk about closed knots, we will say this explicitly.  
\bigskip

The space $\K^n$ is a smooth inifinite-dimensional paracompact manifold \cite[Section 2.2]{KMV:FTHoLinks}, which means that, as explained in \refS{Forms}, we can consider differential forms on it and study their deRham cohomology. 
\bigskip

Classical knot theory ($n=3$) is mainly concerned with computing 
\begin{itemize}
\item $\Ho_0(\K^3)$, which is generated (over $\R$; recall that in this paper, the coefficient ring is always $\R$) by knot types, i.e.~by isotopy classes of knots; and
\item $\Ho^0(\K^3)$, the set of knot invariants, namely locally constant ($\R$-valued) functions on $\K^3$.  These are therefore precisely the functions that take the same value on isotopic knots.
\end{itemize}

The question of computation of knot invariants will be of particular interest to us (see \refS{ConfFTInvs}).   
\bigskip

However, higher (co)homology of $\K^n$ is also interesting, even for $n>3$. Of course, in this case there is no knotting or linking (by a simple general position argument),  so $\Ho^0$ and $\Ho_0$ are trivial, but one can then ask about $\Ho^{>0}$ and $\Ho_{>0}$.  It turns out that our configuration space integrals contain much information about cohomology in various degrees and for all $n>3$ (see \refS{HigherKnots}).


\subsection{Finite type invariants}\label{S:FiniteType}


An interesting set of knot  invariants that our configuration space integrals will produce are \emph{finite type}, or \emph{Vassiliev} invariants.  These invariants are conjectured to \emph{separate} knots, i.e.~to form a \emph{complete} set of invariants.  To explain, there is no known invariant or a set of invariants (that is reasonably computable) with the following property:  

 \begin{quote}
 \emph{Given two non-isotopic knots, there is
 an invariant in this set that takes on different values on these two knots.}
 \end{quote}

The conjecture that finite type invariants form such a set of invariants has been open for some twenty years.  There is some evidence that this might be true since finite type invariants do separate homotopy links \cite{HabLin-Classif} and braids \cite{BN:HoLink, Kohno:LoopsFiniteType} (see \refS{Links} for the definitions of these spaces).
\bigskip

Since finite type invariants will feature prominently in \refS{ConfFTInvs}, we give a brief overview here.  In addition to the separation conjecture, these invariants have received much attention because of their connection to physics (they arise from Chern-Simons Theory), Lie algebras, three-manifold topology, etc. The literature on finite type invariants is abundant, but a good start is  \cite{BN:Vass, CDM:FTInvs}.
\bigskip

Suppose $V$ is a knot invariant, so $V\in\Ho^0(\K^3)$.  Consider the space of \emph{singular links}, which is the subspace of $\Imm_c(\R, \R^n)$ consisting of immersions that are embeddings except for a finite number of double-point self-intersections at which the two derivatives (coming from traveling through the singularity along two different pieces of the knot) are linearly independent.
\bigskip

Each singularity can be locally ``resolved" in two natural ways (up to isotopy), with one strand pushed off the other in one of two directions.  A $k$-singular knot (a knot with $k$ self-intersections) can thus be resolved into $2^k$ ordinary embedded knots.  We can then define $V$ on singular knots as the sum of the values of $V$ on those resolutions, with signs as prescribed in Figure \ref{F:SkeinRelation}.  The expression given in that picture is called the \emph{Vassiliev skein relation} and it depicts the situation locally around a singularity.  The rest of the knot is the same for all three pictures.  The knot should be oriented (by, say, the natural orientation of $\R$); otherwise the two resolutions can be rotated into one another.

\begin{figure}[!htbp]
\begin{center}
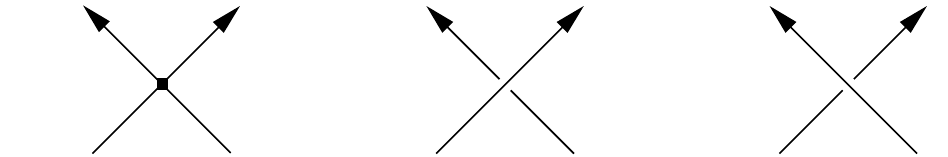
\caption{Vassiliev skein relation.}
\label{F:SkeinRelation}
\end{center}
\end{figure} 

\begin{defi1}
A knot invariant $V$ is \emph{finite type $k$} (or \emph{Vassiliev of type $k$}) if it vanishes on singular knots with $k+1$ self-intersections.  
\end{defi1}

\begin{exam1}
There is only one (up to constant multiple) type 0 invariant, since such an invariant takes the same value on two knots that only differ by a crossing change.  Since all knots are related by crossing changes, this invariant must take the same value on all knots. It is not hard to see that there is also only one type 1 invariant.
\end{exam1}

\begin{exam1}
The coefficients of the Conway, Jones, HOMFLY, and Kauffman polynomials are all finite type invariants \cite{BN:Thesis, BL:Vass}.
\end{exam1}

Let $\calV_k$ be the real vector space generated by all finite type $k$ invariants and let 
$$
\calV=\bigoplus_k \calV_k.
$$
This space is filtered; it is immediate from the definitions that $\calV_k\subset\calV_{k+1}$.
\bigskip

One of the most interesting features of finite type invariants is that a value of a type $k$ invariant $V$ on a $k$-singular knot only depends on the placement of the singularities and not on the immersion itself.  This is due to a simple observation that, if a crossing of a $k$-singular knot is changed, the difference of the evaluation of $V$ on the two knots (before and after the switch) is the value of $V$ on a $(k+1)$-singular knot by the Vassiliev skein relation.  But $V$ is type $k$ so the latter value is zero, and hence $V$ ``does not see" the crossing change.  Since one can get from any singular knot to any other singular knot that has the singularities in the same place (``same place" in the sense that for both knots, there are $2k$ points on $\R$ that are partitioned in pairs the same way; these pairs will make up the $k$ singularities upon the immersion of $\R$ in $\R^3$), $V$ in fact takes the same value on all $k$ singular knots with the same singularity pattern.
\bigskip

The notion of what it means for singularities to be in the ``same place" warrants more explanation and leads to the beautiful and rich connections between finite type invariants and the combinatorics of \emph{chord diagrams} as follows.

\begin{defi1}\label{D:ChordDiagram} A \emph{chord diagram of degree $k$} is a connected graph (one should think of a 1-dimensional cell complex) consisting of an oriented line segment and $2k$ labeled \emph{vertices} marked on it (considered up to orientation-preserving diffeomorphism of the segment).   The graph also contains $k$ oriented \emph{chords} pairing off the vertices (so each vertex is connected to exactly one other vertex by a chord).  
\end{defi1}

We will refer to labels and orientations as \emph{decorations} and, when there is no danger of confusion, we will sometimes draw diagrams without them.
\bigskip

Examples of chord diagrams are given in Figure \ref{F:ChordDiagrams}.  The reader might wish for a more proper combinatorial (rather than descriptive) definition of a chord diagram, and such a definition can be found in \cite[Section 3.1]{KMV:FTHoLinks} (where the  definition is for the case of \emph{trivalent} diagrams which we will encounter below, but it specializes to chord diagrams as the latter are a special case of the former).

\begin{figure}[!htbp]
\begin{center}
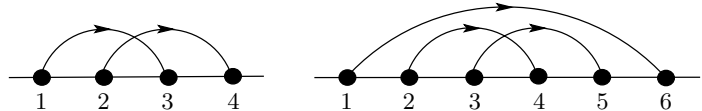
\caption{Examples of chord diagrams. The left one is of degree 2 and the right one is of degree 3. We will always assume the segment is oriented from left to right.}
\label{F:ChordDiagrams}
\end{center}
\end{figure} 

\begin{defi1}\label{D:ChordDiagramSpace}
Define $\mathcal{CD}_k$ to be the real vector space generated by chord diagrams of degree $k$ modulo  the relations
\begin{enumerate}
\item If $\Gamma$ contains more than one chord connecting two vertices, then $\Gamma=0$;
\item  If a diagram $\Gamma$ differs from $\Gamma'$ by a relabeling of vertices or orientations of chords, then 
$$
\Gamma-(-1)^{\sigma}\Gamma'=0
$$
where $\sigma$ is the sum of the order of the permutation of the labels and the number of chords with different orientation;
\item If $\Gamma$ contains a chord connecting two consecutive vertices, then $\Gamma=0$ (this is the \emph{one-term}, or \emph{1T relation});
\item If four diagrams differ only in portions as pictured in Figure \ref{F:4T}, then their sum is 0 (this is the \emph{four-term}, or \emph{4T relation}).
\end{enumerate}
\end{defi1}

\begin{figure}[!htbp]
\begin{center}
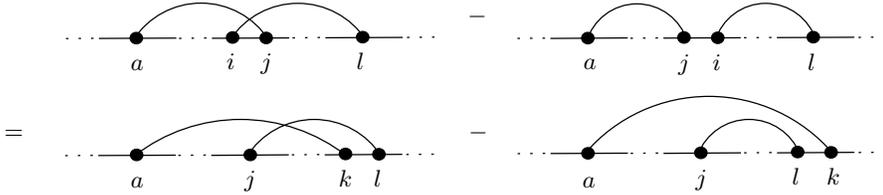
\caption{The four-term (4T) relation.  Chord orientations have been omitted, but they should be the same for all four pictures.}
\label{F:4T}
\end{center}
\end{figure} 

We can now also define the graded vector space
$$
\mathcal{CD}=\bigoplus_k \mathcal{CD}_k.
$$
Let $\mathcal{CW}_k=\Hom(\mathcal{CD}_k, \R)$, the dual of $\mathcal{CD}_k$.  This is called the space of \emph{weight systems of degree $k$}, and is by definition the space of functionals on chord diagrams that vanish on 1T and 4T relations. 
\bigskip

We can now define a function
\begin{align*}
f\colon \calV_k & \longrightarrow \mathcal{CW}_k \\
          V  &  \longrightarrow  
\left( 
\begin{tabular}{rl}
 $W\colon \mathcal{CD}_k$  & $\longrightarrow \R$  \\
 $\Gamma$  & $\longmapsto V(K_{\Gamma})$
 \end{tabular}
 \right)
\end{align*}

where $K_{\Gamma}$ is any $k$-singular knot with singularities as prescribed by $\Gamma$.  By this we mean that there are $2k$ points on $\R$ labeled the same way as in $\Gamma$, and if $x$ and $y$ are points for which there exists a chord in $\Gamma$, then $K_{\Gamma}(x)=K_{\Gamma}(y)$.  The map $f$ is well-defined because of the observation that type $k$ invariant does not depend on the immersion when evaluated on a $k$-singular knot.
\bigskip

The reason that the image of $f$ is indeed in $\mathcal{CW}_k$, i.e.~the reason that the function $W$ we defined above vanishes on the 1T and 4T relations is not hard to see (in fact, 1T and 4T relations are part of the definition of $\mathcal{CD}_k$ precisely because $W$ vanishes on them):   1T relation corresponds to the singular knot essentially having a loop at the singularity; resolving those in two ways results in two isotopic knots on which $V$ has to take the same value (since it is an invariant).  Thus the difference of those values is zero, but by the skein relation, $V$ is then zero on a knot containing such a singularity.  The vanishing on the 4T relation arises from the fact that passing a strand around a singularity of a $(k-1)$-singular knot introduces four $k$-singular knots, and the 4T relation reflects the fact that at the end one gets back to the same $(k-1)$-singular knot.
\bigskip

It is also immediate from the definitions that the kernel of $f$ is precisely type $k-1$ invariants, so that we have an injection (which we will denote by the same letter $f$)
\begin{equation}\label{E:EasyDirection}
f\colon \calV_k/\calV_{k-1} \hookrightarrow \mathcal{CW}_k.
\end{equation}

The following theorem is usually referred to as the \emph{Fundamental Theorem of Finite Type Invariants}, and is due to Kontsevich \cite{K:Fey}.

\begin{thm1}\label{T:Kontsevich}
The map $f$ from equation \eqref{E:EasyDirection} is an isomorphism.
\end{thm1}

Kontsevich proves this remarkable theorem by constructing the inverse to $f$, a map defined by integration that is now known as the \emph{Kontsevich Integral}.  There are now several proof of this theorem, and the one relevant to us gives the inverse of $f$ in terms of configuration space integrals.  See Remark \ref{R:BTProvesKontsevich} for details.
\bigskip

Lastly we describe an alternative space of diagrams that will be better suited for our purposes.

\begin{defi1}\label{D:TrivalentDiagram}
A \emph{trivalent diagram of degree $k$} is a connected graph consisting of an oriented line segment (considered up to orientation-preserving diffeomorphism) and $2k$ labeled vertices of two types: \emph{segment vertices}, lying on the segment, and \emph{free vertices}, lying off the segment.   The graph also contains some number of oriented chords connecting  segment vertices and some number of oriented edges connecting two free vertices or a free vertex and a segment vertex.  Each vertex is trivalent, with the segment adding two to the count of the valence of a segment vertex. 
\end{defi1}

Note that chord diagrams as described in \refD{ChordDiagram} are also trivalent diagrams.  Examples of trivalent diagrams that are not chord diagrams are given in  Figure \ref{F:TrivalentDiagrams}.

\begin{figure}[!htbp]
\begin{center}
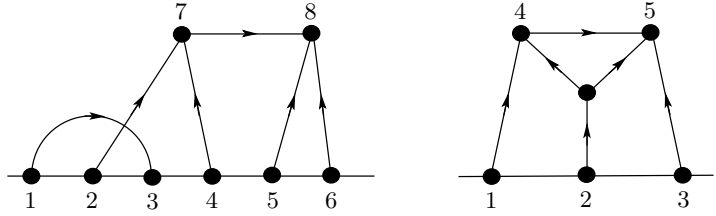
\caption{Examples of trivalent diagrams (that are not chord diagrams). The left one is of degree 4 and the right one is of degree 3.}
\label{F:TrivalentDiagrams}
\end{center}
\end{figure} 

As mentioned before, a more combinatorial definition of trivalent diagrams is given in  \cite[Section 3.1]{KMV:FTHoLinks}.

Analogously to \refD{ChordDiagramSpace}, we have

\begin{defi1}\label{D:TrivalentDiagramSpace}
Define $\mathcal{TD}_k$ to be the real vector space generated by trivalent diagrams of degree $k$ modulo  the relations
\begin{enumerate}
\item If $\Gamma$ contains more than one edge connecting two vertices, then $\Gamma=0$;
\item  If a diagram $\Gamma$ differs from $\Gamma'$ by a relabeling of vertices or orientations of chords or edges, then 
$$
\Gamma-(-1)^{\sigma}\Gamma'=0
$$
where $\sigma$ is the sum of the order of the permutation of the labels and the number of chords and edges with different orientation;
\item If $\Gamma$ contains a chord connecting two consecutive segment vertices, then $\Gamma=0$ (this is same 1T relation from before);
\item If three diagrams differ only in portions as pictured in Figure \ref{F:STU-IHX}, then their sum is 0 (these are called the \emph{STU} and \emph{IHX relation}, respectively.).
\end{enumerate}
\end{defi1}

\begin{figure}[!htbp]
\begin{center}
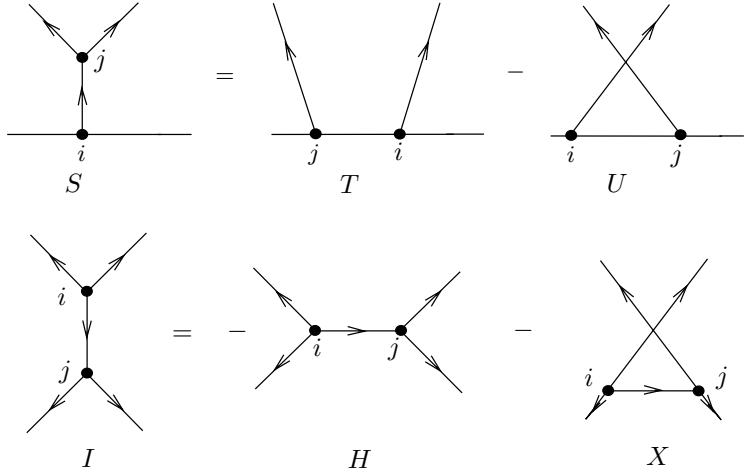
\caption{The STU and the IHX relations.}
\label{F:STU-IHX}
\end{center}
\end{figure} 

Finally let 
$$
\mathcal{TD}=\bigoplus_k \mathcal{TD}_k.
$$

\begin{rem1}
The IHX relations actually follows from the STU relation \cite[Figure 9]{BN:Vass}, but it is important enough that it is usually left in the definition of the space of trivalent diagrams (it gives a connection between finite type invariants and Lie algebras).
\end{rem1}

\begin{thm1}[\cite{BN:Vass}, Theorem 6]\label{T:Chord=Trivalent}
There is an isomorphism 
$$
\mathcal{CD}_k\cong\mathcal{TD}_k.
$$
\end{thm1}

The isomorphism is given by using the STU relation repeatedly to remove all free vertices and turn a trivalent diagram into a chord diagram.  The relationship between the STU and the 4T relation is that the latter is the difference of two STU relations applied to two different edges of the ``tripod" diagram in Figure \ref{F:Tripod}.
\bigskip

One can now again define the space of weight systems of degree $k$ for trivalent diagrams as the space of functionals on $\mathcal{TD}_k$ that vanish on the 1T, STU, and IHX relations.  We will denote these by $\mathcal{TW}_k$.  From \refT{Chord=Trivalent}, we then have
\begin{equation}\label{E:SameWeightSystems}
\mathcal{CW}_k\cong\mathcal{TW}_k.
\end{equation}


\subsection{Configuration spaces and their compactifications}\label{S:Compactification}


At the heart of the construction of our invariants and other cohomology classes of spaces of knots are configuration spaces and their compactifications.

\begin{defi1}\label{D:ConfSpace}
The \emph{configuration space of $p$ points in a manifold $M$} is the space
$$
\Conf(p, M)  =
\{(x_1, x_2, ..., x_{p})\in M^{p} \colon x_i\neq x_j \text{ for } i\neq j\}
$$
\end{defi1}

Thus $\Conf(p, M)$ is $M^{p}$ with all the diagonals, i.e.~the \emph{fat diagonal}, taken out.
\bigskip

We take $\Conf(0, M)$ to be a point.  Since configuration spaces of $p$ points on $\R$ or on $S^1$ have $p!$ components, we take $\Conf(p, \R)$ and $\Conf(p, S^1)$ to mean the component where the points $x_1, ..., x_p$ appear in linear order (i.e.~$x_1<x_2<\cdots<x_p$ on $\R$ or the points are encountered in that order as the circle is traversed counterclockwise starting with $x_1$).

\begin{exam1}\label{Ex:ConfSpExamples}
It is not hard to see that $\Conf(1, \R^n)= \R^n$ and $\Conf(2, \R^n)\simeq S^{n-1}$ (where by ``$\simeq$" we mean homotopy equivalence).  The latter equivalence is given by the \emph{Gauss map}  which gives the direction between the two points:
\begin{align}
\phi\colon \Conf(2, \R^n) & \longrightarrow S^{n-1} \label{E:GaussMap} \\
(x_1,x_2) & \longmapsto \frac{x_1-x_2}{|x_1-x_2|}.\notag
\end{align}
\end{exam1}

We will want to integrate over $\Conf(p, \R^n)$, but this is an open manifold and our integral hence may not converge.  The ``correct" compactification to take, in the sense that it has the same homotopy type as $\Conf(p, \R^n)$ is due to \cite{AS, FM}, and is defined as follows. 
\bigskip 

Recall that that, given a submanifold $N$ of a manifold $M$, the \emph{blowup of $M$ along $N$}, $\operatorname{Bl}(M,N)$, is obtained by replacing $N$ by the unit normal bundle of $N$ in $M$. 

\begin{defi1}\label{D:Compactification}
 Define the \emph{Fulton-MacPherson compactification} of $\Conf(k, M)$, denoted by $\Conf[k,M]$, to be the closure of the image of the inclusion
$$
\Conf(p,M)\longrightarrow M^p\times \prod_{S\subset\{1, ..., p\},\ |S|\geq 2} \operatorname{Bl}(M^S,\Delta_S),
$$
where $M^S$ is the product of the copies of $M$ indexed by $S$ and $\Delta_S$ is the (thin) diagonal in $M^S$, i.e.~the set of points $(x,...,x)$ in $M^S$.
\end{defi1}

Here are some properties of $\Conf[k,M]$ (details and proofs can be found in \cite{S:Compact}):
\begin{itemize}
\item $\Conf[k,M]$ is homotopy equivalent to $\Conf(k, M)$;
 
\item $\Conf[k,M]$ is a manifold with corners, compact when $M$ is compact;
 
\item The boundary of $\Conf[k,M]$ is characterized by points colliding with directions and relative rates of collisions kept track of.  In other words, three points colliding at the same time gives a different point in the boundary than two colliding, then the third joining them;
 
\item The stratification of the boun\-dary is given by stages of co\-lli\-sions of poin\-ts, so if three points collide at the same time, the re\-sul\-ting li\-mi\-ting configuration lies in a codimension one stratum.  If two come to\-ge\-ther and then the third joins them, this gives a point in a co\-di\-men\-sion two stratum since the collision happened in two stages.  In general, a $k$-stage collision gives a point in a co\-di\-men\-sion $k$ stratum.

\item In particular, codimension one boundary of $\Conf[k,M]$ is given by points colliding at the same time.  This will be important in \refS{HigherKnots}, since integration along codimension one boundary is required for Stokes' Theorem.

\item Collisions can be efficiently encoded by different \emph{parenthesizations}, e.g.~the situations described two items ago are parenthesized as $(x_1x_2x_3)$ and $((x_1x_2)x_3)$.  Since parenthesizations are related to trees, the stratification of  $\Conf[k,M]$ can thus also be encoded by trees and this leads to various connections to the theory of \emph{operads} (we will not need this here);

\item $\Conf[k,\R]$ is the \emph{associahedron}, a classical object from homotopy theory;
 
\end{itemize}

Additional discussion of the stratification of $\Conf[p,M]$ can be found in \cite[Section 4.1]{KMV:FTHoLinks}.

\begin{rem1}\label{R:FacesAtInfinity}
Since we will consider configurations on long knots, and these live in $\R^n$, we will think of $\Conf[p,\R^n]$  as the subspace of $\Conf[p+1,S^n]$ where the last point is fixed at the north pole.  Consequently, we will have to consider additional strata given by points escaping to infinity, which corresponds to points colliding with the north pole.
\end{rem1}

\begin{rem1}\label{R:PartialBlowups}
All the properties of the com\-pac\-ti\-fi\-ca\-tions we mentioned hold e\-qua\-lly well in the case where some, but not all, diagonals are blown up.  One can think of constructing the compactification by blowing up the diagonals one at a time, and the or\-der in which we blow up does not matter.  Upon each blow\-up, one ends up with a ma\-ni\-fold with corners.  This ``partial blowup" will be needed in the proof of \refP{ThreeStrataVanishChord} (see also Remark \ref{R:DisconnectedStratum}).
\end{rem1}


\section{Configuration space integrals and finite type knot invariants}\label{S:ConfFTInvs}



\subsection{Motivation:  The linking number}\label{S:LinkingNumber}


Let $\Map_c(\R\sqcup\R, \R^3)$ be the space of smooth maps which are fixed outside some compact set as in the setup leading to \refD{KnotsAndLinks_}  (see \refD{KnotsAndLinks} for details) and define the space of \emph{long (or string) links with two components}, $\Lk_2^3$, as the subspace of $\Map_c(\R\sqcup\R, \R^3)$ given by embeddings.  
\bigskip

Now recall the definition of the configuration space from \refD{ConfSpace} and the Gauss map $\phi$ from \refEx{ConfSpExamples}. Consider the map

\begin{align*}
\Phi\colon  \R\times\R\times\Lk_2^3 & \stackrel{ev}{\longrightarrow}    \Conf(2, \R^3)  \stackrel{\phi}{\longrightarrow}  S^2  \\
 (x_1, x_2, L=(K_1, K_2)) & \longmapsto   (K_1(x_1), K_2(x_2))\longmapsto   \frac{K_2(x_2)-K_1(x_1)}{|K_2(x_2)-K_1(x_1)|}
\end{align*}
So $ev$ is the evaluation map which picks off two points  in $\R^3$, one on each of the strands in the image of a link $L\in\Lk_2^3$, and $\phi$ records the direction between them.  The picture of $\Phi$ is given in Figure \ref{F:LinkingNumber}.

\begin{figure}[!htbp]
\begin{center}
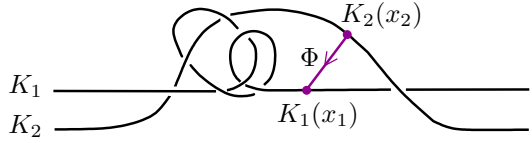
\caption{The setup for the computation of the linking number.}
\label{F:LinkingNumber}
\end{center}
\end{figure} 

Also consider the projection map 
$$
\pi\colon \R\times\R\times\Lk_2^3 \longrightarrow \Lk_2^3
$$
which is a trivial bundle and we can thus perform integration along the fiber on it as described in \refS{Forms}.
\bigskip

Putting these maps together, we have a diagram 
$$
\xymatrix{
 \R\times\R\times\Lk_2^3\ar[r]^-{\Phi}\ar[d]
 ^-{\pi}  &    S^2 \\
 \Lk_2^3  &  
}
$$
which, on the complex of deRham cochains (differential forms), gives a diagram
$$
\xymatrix{
\Omega^*( \R\times\R\times\Lk_2^3)\ar[d]
 ^-{\pi_*}  &    \Omega^*(S^2)\ar[l]_-{\Phi^*} \\
\Omega^{*-2}(\Lk_2^3)  &  
}
$$
Here $\Phi^*$ is the usual pullback and $\pi_*$ is the integration along the fiber.
\bigskip

We will now produce a form on $\Lk_2^3$ by starting with a particular form on $S^2$.  So let $sym_{S^2}\in\Omega^2(S^2)$ be the \emph{unit symetric volume form on $S^2$}, i.e.
$$
sym_{S^2}=\frac{x\, dydz-y\, dxdz+z\, dxdy}{4\pi(x^2+y^2+z^2)^{3/2}}.
$$
This is the form that integrates to 1 over $S^2$ and is rotation-invariant.
\bigskip

Let $\alpha=\Phi^*(sym_{S^2})$.  
\begin{defi1}\label{D:LinkingNumber}
The \emph{linking number} of the link $L=(K_1,K_2)$ is
$$
Lk(K_1,K_2)=\pi_*\alpha=\int_{\R\times\R}\alpha.
$$
\end{defi1}
The expression in this definition is the famous \emph{Gauss integral}.
Since both the form $sym_{S^2}$ and the fiber $\R\times\R$ are two dimensional, the resulting form is 0-dimensional, i.e.~it is a function that assigns a number to each two-component link.  The first remarkable thing is that this form is actually closed, so that the linking number is an element of $\Ho^0(\Lk_2^3)$, an invariant.  The second remarkable thing is that the linking number is actually an integer because it essentially computes the degree of the Gauss map.  Another way to think about this is that the linking number counts the number of times one strand of $L$ goes over the other in a projection of the link, with signs.
\begin{rem1} The reader should not be bothered by the fact that the domain of integration is not compact.  As will be shown in the proof of \refP{ThreeStrataVanishChord}, the integral along \emph{faces at infinity} vanishes.
\end{rem1}


\subsection{``Self-linking" for knots}\label{S:Self-Linking}


One could now try to adapt the procedure that produced the linking number to a single knot in hope of producing some kind of a ``self-linking" knot invariant.
The picture describing this is given in Figure \ref{F:KnotLinking}.

\begin{figure}[!htbp]
\begin{center}
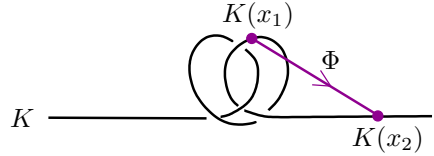
\caption{The setup for the attempted computation of a self-linking number.}
\label{F:KnotLinking}
\end{center}
\end{figure} 

Since the domain of the knot is $\R$, we now take $\Conf(2,\R)$ rather than $\R\times\R$.  Thus the corresponding diagram is 
$$
\xymatrix{
 \Conf(2, \R)\times\K^3\ar[r]^-{\Phi}\ar[d]
 ^-{\pi}  &    S^2 \\
 \K^3  &  
}
$$
The first issue is that the integration over the fiber $\Conf(2, \R)$ may not converge since this space is open.  One potential fix is to use the Fulton-MacPherson compactification from \refD{Compactification} and replace $\Conf(2, \R)$ by $\Conf[2, \R]$.  This indeed takes care of the convergence issue, and we now have a 0-form  whose value on a knot $K\in\K^3$ is
\begin{equation}\label{E:SelfLinkingForm}
A(K)=\int\limits_{\Conf[2, \R]}\left(\frac{K(x_1)-K(x_2)}{|K(x_1)-K(x_2)|}\right)^* sym_{S^2}=\int\limits_{\Conf[2, \R]}\phi^* sym_{S^2}
\end{equation}
(The reason we are denoting this by $A(K)$ is that it will have something to do with the so-caled ``anomalous correction" in \refS{GeneralFiniteType}.)
\bigskip

Checking if this form is closed comes down to checking that \eqref{E:dStokesClosed} is satisfied  ($sym_{S^2}$ is closed so the first term of \eqref{E:dStokes} goes away).  In other words, we need to check that the restriction of the above integral to the face where two points on $\R$ collide vanishes.  However, there is no reason for this to be true.
\bigskip

More precisely, in the stratum where points $x_1$ and $x_2$ in $\Conf[2, \R]$ collide, which we denote by $\partial_{x_1=x_2}\Conf[2, \R]$, the Gauss map becomes the normalized derivative 
\begin{equation}\label{E:DerivativeMap}
\frac{K'(x_1)}{|K'(x_1)|}.
\end{equation}
Pulling back $sym_{S^2}$ via this map to $\partial_{x_1=x_2}\Conf[2, \R]\times \K^3$ and integrating over $\partial_{x_1=x_2}\Conf[2, \R]$, which is now 1-dimensional, produces a 1-form which is the boundary of $\pi_*\alpha$:
$$
d\pi_*\alpha(K)=\int\limits_{\partial_{x_1=x_2}\Conf[2, \R]}\left(\frac{K'(x_1)}{|K'(x_1)|}\right)^* sym_{S^2}.
$$
Since this integral is not necessarily zero, $\pi_*\alpha$ fails to be invariant.
\bigskip

One resolution to this problem is to look for another term which will cancel the contribution of $d\pi_*\alpha$.  As it turns out, that correction is given by the \emph{framing number} \cite{Mos:Self-Linking}.
\bigskip

The strategy for much of what is to come is precisely what we have just seen: We will set up generalizations of this ``self-linking" integration and then correct them with other terms if they fail to give an invariant.
\begin{rem1}
The integral $\pi_*\alpha$ described here is related to the familiar \emph{writhe}.  One way to say why our integral fails to be an invariant is that the writhe is not an invariant -- it fails on the Type I Reidemeister move. 
\end{rem1}


\subsection{A finite type two knot invariant}\label{S:Casson}


It turns out that the next interesting case generalizing the ``self-linking" integral from \refS{Self-Linking} is that of four points and two directions as pictured in Figure \ref{F:KnotTypeTwo}.  

 \begin{figure}[!htbp]
\begin{center}
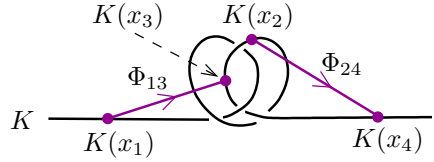
\caption{Toward a generalization of the self-linking number.}
\label{F:KnotTypeTwo}
\end{center}
\end{figure} 

The two maps $\Phi$ now have subscripts to indicate which points are being paired off (the variant where the two maps are $\Phi_{12}$ and $\Phi_{34}$ does not yield anything interesting essentially because of the 1T relation from \refD{ChordDiagramSpace}).  Diagrammatically, the setup can be encoded by the chord diagram $\begin{picture}(0,0)%
\includegraphics{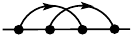}%
\end{picture}%
\setlength{\unitlength}{4144sp}%
\begingroup\makeatletter\ifx\SetFigFont\undefined%
\gdef\SetFigFont#1#2#3#4#5{%
  \reset@font\fontsize{#1}{#2pt}%
  \fontfamily{#3}\fontseries{#4}\fontshape{#5}%
  \selectfont}%
\fi\endgroup%
\begin{picture}(608,156)(980,-6988)
\end{picture}%
$.  This diagram tells us how many points we are evaluating a knot on and which points are being paired off.  This is all the information that is needed to set up the maps
$$
\xymatrix{
 \Conf[4,\R]\times\K^3\ar[rr]^-{\Phi=\Phi_{13}\times\Phi_{24}}\ar[d]
 ^-{\pi}  &   & S^2\times S^2 \\
 \K^3  &  &
}
$$
Let $sym_{S^2}^2$ be the product of two volume forms on $S^2\times S^2$ and let 
$$\alpha=\Phi^*(sym_{S^2}^2).$$
  Since $\alpha$ and $\Conf[4,\R]$, the fiber of $\pi$, are both 4-dimensional, integration along the fiber of $\pi$ thus yields a $0$-form $\pi_*\alpha$ which we will denote by 
$$
(I_{\K^3})_{\begin{picture}(0,0)%
\includegraphics{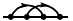}%
\end{picture}%
\setlength{\unitlength}{4144sp}%
\begingroup\makeatletter\ifx\SetFigFont\undefined%
\gdef\SetFigFont#1#2#3#4#5{%
  \reset@font\fontsize{#1}{#2pt}%
  \fontfamily{#3}\fontseries{#4}\fontshape{#5}%
  \selectfont}%
\fi\endgroup%
\begin{picture}(323,89)(980,-6920)
\end{picture}%
}\in\Omega^0(\K^3).
$$
By construction, the value of this form on a knot $K\in\K^3$ is
$$
(I_{\K^3})_{\begin{picture}(0,0)%
\includegraphics{TypeTwoChordSubscript.pdf}%
\end{picture}%
\setlength{\unitlength}{4144sp}%
\begingroup\makeatletter\ifx\SetFigFont\undefined%
\gdef\SetFigFont#1#2#3#4#5{%
  \reset@font\fontsize{#1}{#2pt}%
  \fontfamily{#3}\fontseries{#4}\fontshape{#5}%
  \selectfont}%
\fi\endgroup%
\begin{picture}(323,89)(980,-6920)
\end{picture}%
}(K)=
\int\limits_{\pi^{-1}(K)=\Conf[4,\R]} \alpha
$$
The question now is if this form is an element of $\Ho^0(\K^3)$.  This amounts to checking if it is closed.
 Since $sym_{S^2}$ is closed,  this question reduces by \eqref{E:dStokesClosed} to checking whether the restriction of $\pi_*\alpha$ to the codimension one boundary vanishes:
$$
d(I_{\K^3})_{\begin{picture}(0,0)%
\includegraphics{TypeTwoChordSubscript.pdf}%
\end{picture}%
\setlength{\unitlength}{4144sp}%
\begingroup\makeatletter\ifx\SetFigFont\undefined%
\gdef\SetFigFont#1#2#3#4#5{%
  \reset@font\fontsize{#1}{#2pt}%
  \fontfamily{#3}\fontseries{#4}\fontshape{#5}%
  \selectfont}%
\fi\endgroup%
\begin{picture}(323,89)(980,-6920)
\end{picture}%
}(K)=
\int\limits_{\partial\Conf[4,\R]} \alpha|_{\partial}\, \stackrel{?}{=}0
$$
There is one such boundary integral for each stratum of $\partial\Conf[4,\R]$, and we want the sum of these integrals to vanish.  We will consider various types of faces:
\begin{itemize}
\item \emph{prinicipal faces}, where exactly two points collide;
\item \emph{hidden faces}, where more than two points collide;
\item the \emph{anomalous face}, the hidden face where all points collide (this face will be important later);
\item \emph{faces at infinity}, where one or more points escape to infinity.
\end{itemize}

The principal and hidden faces of $\Conf[4,\R^3]$ can be  encoded by diagrams in Figure \ref{F:BoundaryDiagrams}, which are obtained from diagram \begin{picture}(0,0)%
\includegraphics{TypeTwoChordSmall.pdf}%
\end{picture}%
\setlength{\unitlength}{4144sp}%
\begingroup\makeatletter\ifx\SetFigFont\undefined%
\gdef\SetFigFont#1#2#3#4#5{%
  \reset@font\fontsize{#1}{#2pt}%
  \fontfamily{#3}\fontseries{#4}\fontshape{#5}%
  \selectfont}%
\fi\endgroup%
\begin{picture}(608,156)(980,-6988)
\end{picture}%
 by contracting segments between points (this mimics collisions).  The loop in the three bottom diagrams corresponds to the derivative map since this is exactly the setup that leads to equation \eqref{E:DerivativeMap}.  In other words, for each loop, the map with which we pull back the volume form is the derivative.

\begin{figure}[!htbp]
\begin{center}
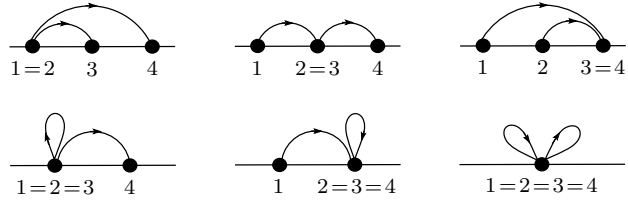
\caption{Diagrams encoding collisions of points.}
\label{F:BoundaryDiagrams}
\end{center}
\end{figure} 

\begin{prop1}\label{P:ThreeStrataVanishChord}
The restrictions of $(I_{\K^3})_{\begin{picture}(0,0)%
\includegraphics{TypeTwoChordSubscript.pdf}%
\end{picture}%
\setlength{\unitlength}{4144sp}%
\begingroup\makeatletter\ifx\SetFigFont\undefined%
\gdef\SetFigFont#1#2#3#4#5{%
  \reset@font\fontsize{#1}{#2pt}%
  \fontfamily{#3}\fontseries{#4}\fontshape{#5}%
  \selectfont}%
\fi\endgroup%
\begin{picture}(323,89)(980,-6920)
\end{picture}%
}$ to hidden faces and faces at infinity vanish.
\end{prop1}

\begin{proof}
 Since there are no maps keeping track of directions between various pairs of points on the knot, for example between $K(x_1)$ and $K(x_2)$ or $K(x_2)$ and $K(x_3)$, the blowup along those diagonals did not need to be performed.  As a result, the stratum where $K(x_1)=K(x_2)=K(x_3)$ is in fact codimension two (three points moving on a one-dimensional manifold became one point) and does not contribute to the integral.  The same is true for the stratum where $K(x_2)=K(x_3)=K(x_4)$.  This is an instance of what is sometimes refered to as the \emph{disconnected stratum}.  The details of why integrals over such a stratum vanish are given in \cite[Proposition 4.1]{V:SBT}.  This argument for vanishing does not work for principal faces since, when two points collide, this gives a codimension one face regardless of whether that diagonal was blown up or not.
 
 For the anomalous face, we have the integral
 $$
 \int\limits_{\partial_A\Conf[4,\R]} \left( \frac{K'(x_1)}{|K'(x_1)|}\times \frac{K'(x_1)}{|K'(x_1)|} \right)^* sym_{S^2}^2.
 $$
The map $\Phi_A$  (the extension of $\Phi$ to the anomalous face) then factors as
$$
\xymatrix{
\partial_A\Conf[4,\R]\times \K^3\ar[rr]^-{\Phi_A}\ar[dr] && S^2\times S^2 \\
&  S^2  \ar[ur]&
}
$$
The pullback thus factors through $S^2$ and, since we are pulling back a 4-form to a 2-dimensional manifold, the form must be zero.

Now suppose a point, say $x_1$, goes to infinity.  The map $\Phi_{13}$ is constant on this stratum so that the extension of $\Phi$ to this stratum again factors through  $S^2$.  If more than one point goes to infinity, the factorization is through a point since both maps are constant.
\end{proof}

\begin{rem1}\label{R:DisconnectedStratum}
It is somewhat strange that one of the vanishing arguments in the above proof required us to essentially go back and change the space we integrate over.  Fortunately, in light of Remark \ref{R:PartialBlowups}, this is not such a big imposition.  The reason this happened is that, when constructing the space $\Conf[4,\R]$, we only paid attention to how many points there were on the diagram $\begin{picture}(0,0)%
\includegraphics{TypeTwoChordSmall.pdf}%
\end{picture}%
\setlength{\unitlength}{4144sp}%
\begingroup\makeatletter\ifx\SetFigFont\undefined%
\gdef\SetFigFont#1#2#3#4#5{%
  \reset@font\fontsize{#1}{#2pt}%
  \fontfamily{#3}\fontseries{#4}\fontshape{#5}%
  \selectfont}%
\fi\endgroup%
\begin{picture}(608,156)(980,-6988)
\end{picture}%
$ and not how they were paired off.  The version of the construction where all the diagram information is taken into account from the beginning is necessary for constructing integrals for homotopy links as will be described in \refS{Links}. 
\end{rem1}

There is however no reason for the integrals corresponding to the principal faces (top three diagrams in Figure \ref{F:BoundaryDiagrams}) to vanish.
One way around this is to look for another space to integrate over which has the same three faces and subtract the integrals. This difference will then be an invariant. To find this space, we again turn to diagrams.  The diagram that fits what we need is given in Figure \ref{F:Tripod} since, when we contract edges to get $4\!=\!1$, $4\!=\!2$, and $4\!=\!3$, the result is same three relevant pictures as before (up to relabeling and orientation of edges). 

\begin{figure}[!htbp]
\begin{center}
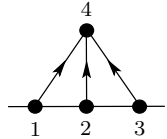
\caption{Diagram whose edge contractions give top three diagrams of Figure \ref{F:BoundaryDiagrams}.}
\label{F:Tripod}
\end{center}
\end{figure} 

The picture suggests that we want a space of four configuration points in $\R^3$, three of which lie on a knot, and we want to keep track of three directions between the points on the knot and the one off the knot (since the diagram has those three edges).
In other words, we want the situation from Figure \ref{F:KnotTripod}.

 \begin{figure}[!htbp]
\begin{center}
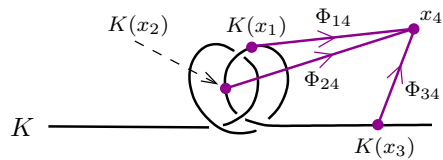
\caption{The situation schematically given by the diagram from Figure \ref{F:Tripod}.}
\label{F:KnotTripod}
\end{center}
\end{figure} 

To make this precise, consider the pullback space

\begin{equation}\label{E:BTPullback}
\xymatrix{
\Conf[3,1;\K^3, \R^3] \ar[r] \ar[d] &  
\Conf[4, \R^3] \ar[d]^-{proj} \\
\Conf[3,\R]\times \K^3 \ar[r]^-{eval}  &  \Conf[3,\R^3]
}
\end{equation}
where $eval$ is the evaluation map and $proj$ the projection onto the first there points of a configuration.
There is now an evident map
$$\pi'\colon \Conf[3,1;\K^3, \R^3]\longrightarrow \K^3$$
whose fiber over $K\in\K^3$ is precisely the configuration space of four points, three of which are constrained to lie on $K$.

\begin{prop1}[\cite{BT}]
 The map $\pi'$ is a smooth bundle whose fiber is a finite-dimensional smooth manifold with corners.
\end{prop1}

This allows us to perform integration along the fiber of $\pi'$.  So let
$$
\Phi=\Phi_{14}\times\Phi_{24}\times\Phi_{34}\colon \Conf[3,1;\K^3, \R^3]\longrightarrow
(S^2)^3
$$
be the map giving the three directions as in Figure \ref{F:KnotTripod}. The relevant maps are thus
$$
\xymatrix{
 \Conf[3,1;\K, \R^3]\ar[r]^-{\Phi}\ar[d]
 ^-{\pi'}   & (S^2)^3 \\
 \K  &  
}
$$
As before, let $\alpha'=\Phi^*(sym_{S^2}^3)$.  This form can be integrated along the fiber $\Conf[3,1; K, \R^3]$  over $K$.  Notice that both the form and the fiber are now 6-dimensional, so integration gives a form
$$
(I_{\K^3})_{\begin{picture}(0,0)%
\includegraphics{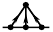}%
\end{picture}%
\setlength{\unitlength}{4144sp}%
\begingroup\makeatletter\ifx\SetFigFont\undefined%
\gdef\SetFigFont#1#2#3#4#5{%
  \reset@font\fontsize{#1}{#2pt}%
  \fontfamily{#3}\fontseries{#4}\fontshape{#5}%
  \selectfont}%
\fi\endgroup%
\begin{picture}(238,138)(980,-6729)
\end{picture}%
}\in\Omega^0(\K^3).
$$
The value of this form on $K\in\K^3$ is thus 
$$
(I_{\K^3})_{\begin{picture}(0,0)%
\includegraphics{TripodSubscript.pdf}%
\end{picture}%
\setlength{\unitlength}{4144sp}%
\begingroup\makeatletter\ifx\SetFigFont\undefined%
\gdef\SetFigFont#1#2#3#4#5{%
  \reset@font\fontsize{#1}{#2pt}%
  \fontfamily{#3}\fontseries{#4}\fontshape{#5}%
  \selectfont}%
\fi\endgroup%
\begin{picture}(238,138)(980,-6729)
\end{picture}%
}(K)=
\int\limits_{(\pi')^{-1}(K)=\Conf[3,1;,K,\R^3]} \alpha'
$$

We now have the analog of \refP{ThreeStrataVanishChord}:

\begin{prop1}\label{P:ThreeStrataVanishTripod}
The restrictions of $(I_{\K^3})_{\begin{picture}(0,0)%
\includegraphics{TripodSubscript.pdf}%
\end{picture}%
\setlength{\unitlength}{4144sp}%
\begingroup\makeatletter\ifx\SetFigFont\undefined%
\gdef\SetFigFont#1#2#3#4#5{%
  \reset@font\fontsize{#1}{#2pt}%
  \fontfamily{#3}\fontseries{#4}\fontshape{#5}%
  \selectfont}%
\fi\endgroup%
\begin{picture}(238,138)(980,-6729)
\end{picture}%
}$ to the hidden faces and faces at infinity vanish.  The same is true for the two principal faces given by the collisions of two points on the knot.
\end{prop1} 

\begin{proof}
The same arguments as in \refP{ThreeStrataVanishChord} work here, although an alternative is possible:  For any of the hidden faces or the two principal faces from the statement of the proposition, at least two of the maps will be the same.  For example, if $K(x_1)=K(x_2)$, then $\Phi_{14}=\Phi_{24}$ and $\Phi$ hence factors through a space of strictly lower dimension than the dimension of the form.  

A little more care is needed for faces at infinity.  If some, but not all points go to infinity (this includes $x_4$ going to infinity in any direction), the same argument as in \refP{ThreeStrataVanishChord} holds.  If all points go to infinity, then it can be argued that, yet again, the map $\Phi$ factors through a space of lower dimension, but this has to be done a little more carefully; see proof of \cite[Proposition 4.31]{KMV:FTHoLinks} for details. 
\end{proof}

We then have
 
\begin{thm1}\label{T:Casson}
The map 
\begin{align*}
\K & \longrightarrow \R \\
K & \longmapsto \Big((I_{\K^3})_{\begin{picture}(0,0)%
\includegraphics{TypeTwoChordSubscript.pdf}%
\end{picture}%
\setlength{\unitlength}{4144sp}%
\begingroup\makeatletter\ifx\SetFigFont\undefined%
\gdef\SetFigFont#1#2#3#4#5{%
  \reset@font\fontsize{#1}{#2pt}%
  \fontfamily{#3}\fontseries{#4}\fontshape{#5}%
  \selectfont}%
\fi\endgroup%
\begin{picture}(323,89)(980,-6920)
\end{picture}%
}(K)-
(I_{\K^3})_{\begin{picture}(0,0)%
\includegraphics{TripodSubscript.pdf}%
\end{picture}%
\setlength{\unitlength}{4144sp}%
\begingroup\makeatletter\ifx\SetFigFont\undefined%
\gdef\SetFigFont#1#2#3#4#5{%
  \reset@font\fontsize{#1}{#2pt}%
  \fontfamily{#3}\fontseries{#4}\fontshape{#5}%
  \selectfont}%
\fi\endgroup%
\begin{picture}(238,138)(980,-6729)
\end{picture}%
}(K)\Big)
\end{align*}
is a knot invariant, i.e.~an element of $\Ho^0(\K^3)$.  Further, it is a finite type two invariant.
\end{thm1}

This theorem was proved in \cite{GMM} and in \cite{BN:Thesis}.  Since the set up there was for closed knots, the diagrams were based on circles, not segments, and one hence had to include some factors in the above formula having to do with the automorphisms of those diagrams.  We will encounter  automorphism factors like these in \refT{AllFiniteType}.
\bigskip

Recall the discussion of finite type invariants from \refS{FiniteType}.  The invariant from \refT{Casson} turns out to be the unique  finite type two invariant which takes the value of zero on the unknot and one on the trefoil.  It is also equal to the coefficient of the quadratic term of the Conway polynomial, and it is equal to the Arf invariant when reduced mod 2.  In addition, it appears in the surgery formula for the Casson invariant of homology spheres and is thus also known as the Casson \emph{knot} invariant.   A treatment of this invariant from the Casson point of view can be found in \cite{PV:Casson}.

\begin{proof}[Proof of \refT{Casson}]
In light of Propositions \ref{P:ThreeStrataVanishChord} and \ref{P:ThreeStrataVanishTripod}, the only thing to show is that the integrals along the principal faces match.  This is clear essentially from the pictures (since the diagram pictures representing those collisions are the same), except that the collisions of $K(x_i)$, $i=1,2,3$, with $x_4$ produce an extra map on each face (extension of $\Phi_{i4}$ to that face) that is not present in the first integral.  Since $K(x_4)$ can approach $K(x_i)$ from any direction, the extension of $\Phi_{i4}$ sweeps out a sphere, and this is independent of the other two maps.  By Fubini's Theorem, we then have for, say, the case $i=1$,
\begin{align*}
&     \int\limits_{\partial_{K(x_1)=x_4}\Conf[3,1;K,\R^3]}( \Phi_{14}\times\Phi_{24}\times\Phi_{34})^* sym_{S^2}^3 \\
= &   \int\limits_{S^2}(\Phi_{14})^*sym_{S^2}
\cdot
\int\limits_{\partial_{K(x_1)=x_4}\Conf[3,1;K,\R^3]}(\Phi_{24}\times\Phi_{34} )^* sym_{S^2}^2 \\
= & \int\limits_{\Conf[3,\R]} (\Phi_{12}\times\Phi_{13})^* sym_{S^2}^2
\end{align*}
The last line is obtained by observing that the first integral in the previous line is 1 (since $sym_{S^2}$ is a unit volume form) and by rewriting the domain $\partial_{K(x_1)=x_4}\Conf[3,1;K,\R^3]$ as $\Conf[3,\R]$ (and relabeling the points).  The result is precisely the integral of one of the principal faces of $(I_{\K^3})_{\begin{picture}(0,0)%
\includegraphics{TypeTwoChordSubscript.pdf}%
\end{picture}%
\setlength{\unitlength}{4144sp}%
\begingroup\makeatletter\ifx\SetFigFont\undefined%
\gdef\SetFigFont#1#2#3#4#5{%
  \reset@font\fontsize{#1}{#2pt}%
  \fontfamily{#3}\fontseries{#4}\fontshape{#5}%
  \selectfont}%
\fi\endgroup%
\begin{picture}(323,89)(980,-6920)
\end{picture}%
}$.  The remaining two principal faces can similarly be matched up.  Some care should be taken with signs, and we leave it to the reader to check those, keeping in mind that relabeling the vertices of a diagram may introduce a sign in the integral (this corresponds to permuting the copies of $\R$ and $\R^3$ and, since the dimensions of these spaces are odd, this in turn preserves or reverses the orientation of the fiber depending on the sign of the permutation).  Changing orientations of chords of edges also might affect the sign (this corresponds to composing with the antipodal map to $S^2$ which changes the sign of the pullback form).
\bigskip

To show that this is a finite type two invariant is not difficult.  The key is that the resolutions of the three singularities can be chosen as small as desired.  Then the integration domain can be broken up into subsets on which the difference of the integrals between the two resolutions is zero.  Details can be found in \cite[Section 5]{V:SBT}.
\end{proof}


\subsection{Finite type $k$ knot invariants}\label{S:GeneralFiniteType}


Recall the space of trivalent diagrams from \refD{TrivalentDiagramSpace}.
The two diagrams appearing in \refT{Casson} are the two (up to decorations) elements of $\mathcal{TD}_2$.  However, the recipe for integration that these diagrams gave us in the previous section generalizes to any diagram. Namely, any diagram $\Gamma\in\mathcal{TD}_k$ with $p$ segement vertices and $q$ free vertices gives a prescription for constructing a pullback as in \eqref{E:BTPullback}:

\begin{equation}\label{E:pqpullback}
\xymatrix{
\Conf[p,q;\K^3, \R^3] \ar[r] \ar[d] &  
\Conf[p+q, \R^3] \ar[d]^-{proj} \\
\Conf[p,\R]\times \K^3 \ar[r]^-{eval}  &  \Conf[p,\R^3]
}
\end{equation}
There is again a bundle
$$
\pi\colon \Conf[p,q;\K^3, \R^3]\longrightarrow \K^3
$$
whose fibers are manifolds with corners.
We also have a map
$$
\Phi\colon \Conf[p,q;\K^3, \R^3]\longrightarrow (S^2)^e
$$
where 
\begin{itemize}
\item $\Phi$ is the product of the Gauss maps between pairs of configuration points corresponding to the edges of $\Gamma$, and 
\item $e$ is the number of chords and edges of $\Gamma$.
\end{itemize}

Let $\alpha=\Phi^*(sym_{S^2}^e)$.  It is not hard to see that, because of the trivalence condition on $\mathcal{TD}_k$, the dimension of the fiber of $\pi$ is $2e$, as is the dimension of $\alpha$.  Thus we get a 0-form $\pi_*\alpha$, or, in the notation of \refS{Casson}, a form
$$
(I_{\K^3})_{\Gamma}\in\Omega^0(\K^3),
$$
whose value on $K\in\K^3$ is 
$$
(I_{\K^3})_{\Gamma}(K)=
\int\limits_{\pi^{-1}(K)=\Conf[p,q; K, \R^3]} \alpha.
$$

Now recall from discussion preceeding \eqref{E:SameWeightSystems} that $\mathcal{TW}_k$ is the space of weight systems of degree $k$, i.e.~functionals on $\mathcal{TD}_k$.  Also recall the ``self-linking" integral from equation \eqref{E:SelfLinkingForm}.  Finally let $\mathcal{TB}_k$ be a basis of diagrams for $\mathcal{TD}_k$ (this is finite and canonical up to sign) and let $|\operatorname{Aut}(\Gamma)|$ be the number of automorphisms of $\Gamma$ (these are automorphisms that fix the segment, regarded up to labels and edge orientations).  

\begin{thm1}[\cite{Thurs}]\label{T:AllFiniteType}
For each $W\in\mathcal{TW}_k$, the map 
\begin{align}
 \K^3 & \longrightarrow \R \label{E:ThurstonMap}\\
K & \longmapsto \sum_{\Gamma\in\mathcal{TB}_k}
\left(
\frac{W(\Gamma)}{|\operatorname{Aut}(\Gamma)|}(I_{\K^3})_{\Gamma}
-\mu_{\Gamma}A(K)
\right), \notag
\end{align}
where $\mu_{\Gamma}$ is a real number that only depends on $\Gamma$, is a finite type $k$ invariant.  Furthermore, the assignment  $W\mapsto V\in \calV_k$ gives an isomorphism
$$
I_{\K^3}^0\colon\mathcal{TW}_k\longrightarrow \mathcal{V}_k/\mathcal{V}_{k-1}
$$
(where the map \eqref{E:ThurstonMap} is followed by the quotient map $\mathcal{V}_k\to\mathcal{V}_k/\mathcal{V}_{k-1}$).
\end{thm1}

\begin{proof} The argument here is essentially the same as in \refT{Casson} but is complicated by the increased number of cases.  Once again, to start, one should observe that the relations in \refD{TrivalentDiagramSpace} are compatible with the sign changes that occur in the integral if copies of $\R$ or $\R^3$ in the bundle $\Conf[p,q;\K^3,\R^3]$ are switched (the orientation of this space would potentially change and so would the sign of the integral) or if a Gauss map is replaced by its antipode.
\bigskip

To prove that the integrals along hidden faces vanish, one considers various types of faces as in Propositions \ref{P:ThreeStrataVanishChord} and \ref{P:ThreeStrataVanishTripod}.  If the points that are colliding form a ``disconnected stratum" in the sense that the set of vertices and edges of the corresponding part of $\Gamma$ forms two subsets such that no chord of edge connects a vertex of one subset to a vertex of the other, we revise the definition of $\Conf[p,q;\K^3,\R^3]$ and turn this stratum into one of codimension greater than one.  This takes care of most hidden faces \cite[Section 4.2]{V:SBT}.  The remaining ones are disposed of by symmetry arguments due to Kontsevich \cite{K:Vass} (see also \cite[Section 4.3]{V:SBT}) which show that some of the integrals are equal to their negatives and thus vanish.  Another reference for the vanishing along hidden faces is \cite[Theorem A.6]{CCRL} (the authors of that paper consider closed knots but this does not change the arguments).

The vanishing of the integrals along faces at infinity goes exactly the same way as in \refP{ThreeStrataVanishChord}; the map $\Conf[p,q;\K^3,\R^3]\to (S^2)^e$ always factors through a space of lower dimension.  More details can be found in \cite[Section 4.5]{V:SBT}. 

Lastly, the vanishing of principal faces occurs due to the STU and IHX relations.  The STU case is provided in Figure \ref{Fig:STUCancelation} (we have omitted the labels on diagrams and signs to simplify the picture).

\begin{figure}[!htbp]
  \scalebox{0.85}{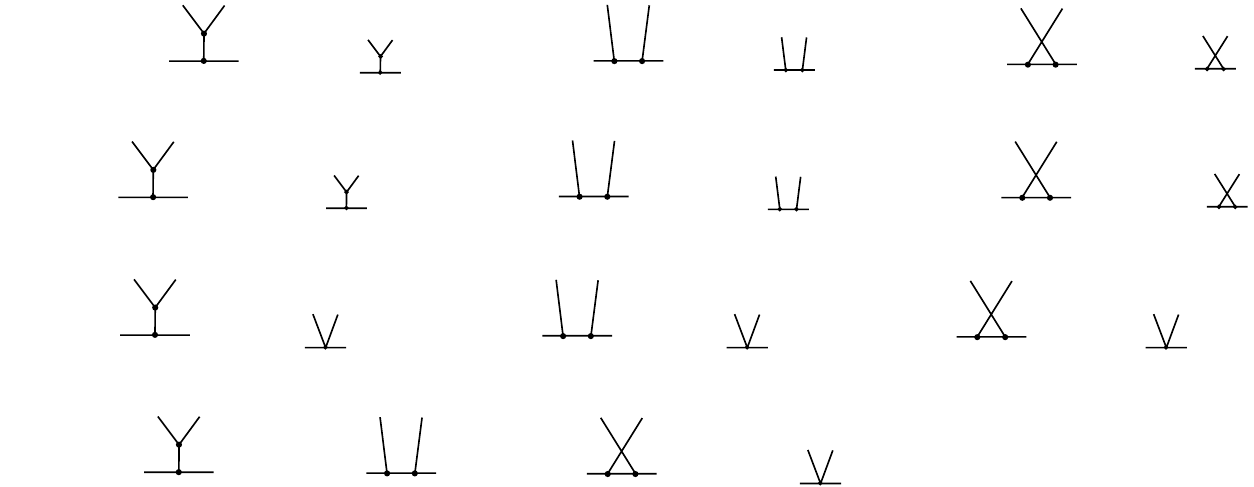}
\caption{Cancellation due to the STU relation}
\label{Fig:STUCancelation}
\end{figure}

Similar cancellation occurs with principal faces resulting from collision of free vertices, where one now uses the IHX relation.  The contributions from all principal faces thus cancel.
\bigskip

The one boundary integral that is not know to vanish (but is conjectured to; it is known that it does in a few low degree cases) is that of the anomalous face corresponding to all points colliding.  It turns out that this integral is some multiple $\mu_{\Gamma}$ of the self-linking integral $A(K)$, and hence the term $\mu_{\Gamma}A(K)$ is subtracted so that we get an invariant.  For further details, see \cite[Section 4.6]{V:SBT}.
\bigskip

To show that this invariant is finite type $k$ and that we get an isomorphism $\mathcal{TW}_k\to\mathcal{V}_k/\mathcal{V}_{k-1}$ is tedious but straightforward.  The point is that, as in the proof of \refT{Casson}, the resolutions of the $k+1$ singularities can be chosen to differ in arbitratily small neighborhoods and so the integrals from the sum of \eqref{E:ThurstonMap} cancel out.  For this to happen, the domain of integration is  subdivided and the integrals end up pairing off and canceling on appropriate neighborhoods.  For details on how this leads to the conclusion that the invariant is finite type, see \cite[Lemma 5.4]{V:SBT}.  Finally, it is easy to show that the map $I_{\K^3}^0$, when composed with the isomorphism \eqref{E:SameWeightSystems}, is an inverse to the map $f$ from equation \eqref{E:EasyDirection}, and this gives the desired isomorphism \cite[Theorem 5.3]{V:SBT}.
\end{proof}

\begin{rem1}\label{R:BTProvesKontsevich}
In combination with \eqref{E:SameWeightSystems}, \refT{AllFiniteType} thus gives an alternative proof of Kontsevich's \refT{Kontsevich}.  
\end{rem1}

\begin{rem1}
The reason we put a superscript ``0" on the map $I_{\K^3}^0$ is that this is really just the degree zero manifestation of a chain map described in \refS{HigherKnots}.
\end{rem1}

\begin{rem1}
In \refT{Casson}, there is one weight system for the degree two case and it takes one the values $1$ and $-1$ for the two relevant diagrams.  In addition, the anomalous correction in that case vanishes, so \refT{AllFiniteType} is indeed a generalization of \refT{Casson}.
\end{rem1}

\refT{AllFiniteType} thus gives a way to construct all finite type invariants.  In addition, configuration space integrals provide an important link between Chern-Simons Theory (where the first instances of such integrals occur), topology, and combinatorics.  Unfortunately, computations with these integrals are difficult and only a handful have been performed.


\section{Generalization to $\K^n$, $n>3$}\label{S:HigherKnots}


Just as there was no reason to stop at diagrams with four vertices, there is no reason to stop at trivalent diagrams.  One might as well take diagrams that are at least trivalent, such as the one from Figure \ref{F:GeneralKnotDiagram} (less than trivalent turns out not to give anything useful).

\begin{figure}[!htbp]
\begin{center}
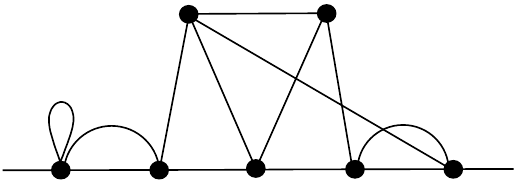
\caption{A more general diagram (without decorations).}
\label{F:GeneralKnotDiagram}
\end{center}
\end{figure} 

To make this precise, we generalize \refD{TrivalentDiagram}  as follows:

\begin{defi}\label{D:Diagrams}  For $n\geq 3$, 
define a \emph{diagram} to be a connected graph consisting of an oriented line segment (considered up to orientation-preserving diffeomorphism) and some number of vertices of two types: \emph{segment vertices}, lying on the segment, and \emph{free vertices}, lying off the segment.  The graph also contains some number of 
\begin{itemize}
\item \emph{chords} connecting distinct segment vertices;
\item  \emph{loops} connecting segment vertices to themselves; and  
\item \emph{edges} connecting two free vertices or a free vertex and a segment vertex.
\end{itemize}  Each vertex is at least trivalent, with the segment adding two to the count of the valence of a segment vertex.  In addition,
\begin{itemize}
\item if $n$ is odd, all vertices are labeled, and edges, loops, and chords are oriented;
\item if $n$ is even, external vertices, edges, loops, and chords are labeled.
\end{itemize}
\end{defi}

We also identify \emph{arcs}, which are parts of the segment between successive segment vertices.

\begin{defi}\label{D:DiagramSpace}
For each $n\geq 3$, let $\mathcal{D}^n$ be the real vector space generated by diagrams from \refD{Diagrams} 
 modulo  the relations
\begin{enumerate}
\item If $\Gamma$ contains more than one edge connecting two vertices, then $\Gamma=0$;
\item  If $n$ is odd and if a diagram $\Gamma$ differs from $\Gamma'$ by a relabeling of vertices or orientations of loops, chord, and edges, then 
$$
\Gamma-(-1)^{\sigma}\Gamma'=0
$$
where $\sigma$ is the sum of the sign of the permutation of vertex labels and the number or loops, chords, and edges with different orientation.
\item  If $n$ is even and if a diagram $\Gamma$ differs from $\Gamma'$ by a relabeling of segment vertices or loops, chord, and edges, then 
$$\Gamma=(-1)^\sigma\Gamma',
$$
where $\sigma$ is the sum of the signs of these permutations.
\end{enumerate}
\end{defi}

Note that $\mathcal{TD}_k$ is the quotient of the subspace of $\calD^n$ generated by trivalent diagrams with $2k$ vertices.  

\begin{defi}\label{D:Degree}
Define the \emph{degree} of $\Gamma\in\mathcal{D}^n$ to be
$$
\deg(\Gamma)=2(\#\ \text{edges})-3(\#\ \text{free vertices})-(\#\ \text{segment vertices}).
$$
\end{defi}

Thus if $\Gamma$ is a trivalent diagram,  $\deg(\Gamma)=0$.
\bigskip

Coboundary $\delta$ is given on each diagram by contracting edges and segments (but not chords or loops since this does not represent a collision of points).  Namely, let $e$ be an edge or an arc of $\Gamma$ and let $\Gamma/e$ be $\Gamma$ with $e$ contracted.  Then 
$$
\delta(\Gamma)=\sum_{\text{edges and arcs $e$ of $\Gamma$}}\epsilon(\Gamma)\Gamma/e,
$$ 
where $\epsilon(\Gamma)$ is a sign determined by
\begin{itemize}
\item if $n$ is odd and $e$ connects vertex $i$ to vertex $j$, $\epsilon(\Gamma)=(-1)^j$ if $j>i$ and $\epsilon(\Gamma)=(-1)^{i+1}$ if $j<i$;
\item if $n$ is even and $e$ is an arc connecting vertex $i$ to vertex $j$, then $\epsilon(\Gamma)$ is computed as above, and if $e$ is an edge, then $\epsilon(\Gamma)=(-1)^s$, where $s=\text{(label of $e$)}+(\#\ \text{segment vertices})+1$.
\end{itemize}

On $\calD^n$, $\delta$ is the linear extension of this.
An example (without decorations and hence modulo signs) is given in Figure \ref{F:BoundaryDiagramsAgain}.

\begin{figure}[!htbp]
\begin{center}
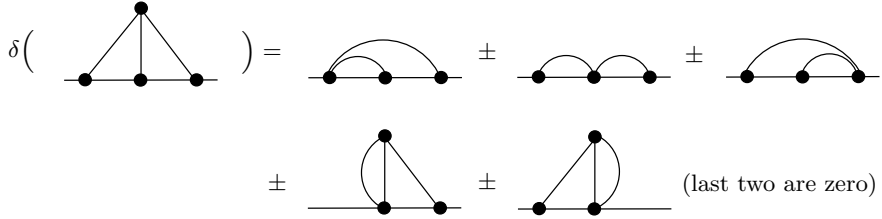
\caption{An example of a coboundary.}
\label{F:BoundaryDiagramsAgain}
\end{center}
\end{figure} 

\begin{thm}[\cite{CCRL}, Theorem 4.2]\label{T:Coboundary}
For $n\geq 3$,  $
(\mathcal{D}^n, \delta)
$
is a cochain complex.
\end{thm}

\begin{proof}
It is a straightforward combinatorial exercise to see that $\delta$ raises degree by 1 and that $\delta^2=0$. 
\end{proof}

Returning to configuration space integrals, for each $\Gamma\in\mathcal D^n$ and $K\in\K^n$, we can still define an integral  as before.   Vertices of $\Gamma$ tell us what pullback bundle $\Conf[p,q;\K^n, \R^n]$ to construct, i.e.~how many points to have on and off the knot.  The only difference is that the map 
$$
\Phi\colon \Conf[p,q;\K^n, \R^n]\longrightarrow (S^{n-1})^{(\#\ \text{loops, chords, and edges of $\Gamma$})}
$$ 
 is now a product of Gauss maps for each edge and chord of $\Gamma$ as well as the derivative map for each loop of $\Gamma$.  Further,  none of what was done in \refS{GeneralFiniteType} requires $n=3$.  More precisely, we still get a form, for $n\geq 3$,
$$
(I_{\K^n})_{\Gamma}\in \Omega^d(\K^n),
$$
given by integration along the fiber of the bundle
$$
\pi\colon \Conf[p,q;\K^n, \R^n]\longrightarrow \K^n.
$$
The degree of the form is no longer necessarily zero but of degree equal to
\begin{align*}
& (n-1)(\#\ \text{loops, chords, and edges of $\Gamma$})\\
 -& n(\text{\#\ free vertices of $\Gamma$})\\
 - & (\text{\#\ segment vertices of $\Gamma$})
\end{align*}
and its value on $K\in\K^n$ is as before 
$$
(I_{\K^n})_{\Gamma}(K)=
\int\limits_{\pi^{-1}(K)=\Conf[p,q; K, \R^n]} \alpha.
$$ 

\begin{thm}[\cite{CCRL}, Theorem 4.4]\label{T:Italians}
For $n>3$, configuration space integrals give a cochain map
$$
I_{\K^n}\colon (\mathcal{D}^n, \delta) \longrightarrow (\Omega^*(\K^n), d).
$$
\end{thm}

(Here $d$ is the ordinary deRham differential.)

\begin{proof}
To prove this, one first observes that, just as in \refT{AllFiniteType}, the relations in $\mathcal D^n$ correspond precisely to what happens on the integration side if points or maps are permuted or if Gauss maps are composed with the antipodal map.  In fact, those relations in $\mathcal D^n$ are defined precisely because of what happens on the integration side.

Once this is established, it is necessary to show that, for each $\Gamma$, the integrals along the hidden faces and faces at infinity vanish, and this goes exactly the same way as in \refT{AllFiniteType}. The integrals along principal faces correspond precisely to contractions of edges and arcs, so that the map commutes with the differential.
\end{proof}

 \begin{rem}
 There is an algebra structure on $\calD^n$ given by the \emph{shuffle product} that is compatible with the wedge product of forms \cite{CCRL:Struct}.  This makes  $I_{\K^n}$ a map of algebras as well.
 \end{rem}
 
By evaluating  $I_{\K^n}$ on certain diagrams, Cattaneo, Cotta-Ramusino, and Longoni \cite{CCRL} also prove
 
\begin{cor}
Given any $i>0$, the knot space $\K^n$, $n>3$, has nontrivial cohomology in dimension greater than $i$.
\end{cor}

Complex $\calD^n$ is known to have the same homology as the $\K^n$, so it contains a lot of information about the topology of long knots.  However, we do not know that this map induces an isomorphism.  More precisely, we have

\begin{conj}
The map $I_{\K^n}$ is a quasi-isomorphism.
\end{conj}

Even though we do not have \refT{Italians} for $n=3$, the construction is compatible with what we already did in the case of classical knots $\K^3$.  Namely, for $n=3$, one does not get a cochain map in all degrees because of the anomalous face. But in degree zero, it turns out that
$$
\Ho^0(\mathcal{D}^3)=\mathcal{TD}
$$
(up to certain diagram automorphism factors; see \cite[Section 3.4]{KMV:FTHoLinks}).
In other words, the kernel of $\delta$ in degree zero is defined by imposing the 1T, STU, and IHX relations.
 \bigskip
 
Thus, after correcting by the anomalous correction and after identifying $\mathcal{TD}$ with its dual, the space of weight systems $\mathcal{TW}$ (the dualization gymnastics is described in \cite[Section 3.4]{KMV:FTHoLinks}), we get a map
$$
(\Ho^0(\mathcal{D}^3))^*=\mathcal{TW}\longrightarrow \Ho^0(\K).
$$
But this is precisely the map $I_{\K^3}^0$ from \refT{AllFiniteType} and we already know that the image of this map is the finite type knot invariants.


\section{Further generalizations and applications}\label{S:Generalizations}


 In this section we give brief overviews of other contexts in which  configuration space integrals have appeared in recent years.


\subsection{Spaces of links}\label{S:Links}


Configuration space integrals can also be defined for spaces of long links, homotopy links, and braids.  (The reader should keep in mind that it is actually quite surprising that they can be defined for homotopy links.)  The results stated here encompass those for knots (by setting $m=1$ in $\Lk_m^n$).
\bigskip

For $m\geq 1$, $n\geq 2$, let $\Map_c(\sqcup_m\R, \R^n)$ be the space of smooth maps of $\sqcup_m\R$ to $\R^n$ which, outside of $\sqcup_m I$ agree with the map $\sqcup_m\R\to \R^n$, which is on the $i$th copy of $\R$ given as
$$
t\longmapsto (t,i,0,0,...,0).
$$
As in the case of knots (\refD{KnotsAndLinks_}), we can define the spaces of links as subspaces of $\Map_c(\sqcup_m\R, \R^n)$ with the induced topology as follows.  

\noindent\begin{defi1}\label{D:KnotsAndLinks}
Define the space of 
\begin{itemize} 
\item \emph{long (or string) links with $m$ strands} $\Lk_m^n\subset \Map_c(\sqcup_m\R, \R^n)$ as the space of embeddings $L\colon \sqcup_m\R\to \R^n$.
\item \emph{pure braids on $m$ strands} $\Br_m^n\subset \Map_c(\sqcup_m\R, \R^n)$ as the space of embeddings $B\colon \sqcup_m\R\to \R^n$ whose derivative in the direction of the first coordinate is positive.
\item \emph{long (or string) homotopy links with $m$ strands} $\HLk_m^n\subset \Map_c(\sqcup_m\R,$ $ \R^n)$ as the space of link maps  $H\colon \sqcup_m\R\to \R^n$ (smooth maps of $\sqcup_m\R$ in $\R^n$ with the images of the copies of $\R$ disjoint).
\end{itemize}
\end{defi1}

\begin{rem1}
Another (and in fact, more standard) way to think about $\Br_m^n$ is as the loop space  $\Omega\Conf(m,\R^{n-1})$. 
\end{rem1}

\begin{rem1}
For technical reasons, it is sometimes better to take strands that are not parallel outside of $I^n$, but this does not change anything about the theorems described here.  For details, see \cite[Definition 2.1]{KMV:FTHoLinks}.
\end{rem1}

Some observations about these spaces are:

\begin{itemize}
\item $\Br_m^n\subset \Lk_m^n \subset \HLk_m^n$;  
\item In $\pi_0(\HLk_m^n)$, we can pass a strand through itself so this can be thought of as space of ``links without knotting";
\end{itemize}

Example of a homotopy link and a braid is given in Figure \ref{F:LinkExamples}.  Note that, as usual, we have confused the maps $H$ and $L$ with their images in $\R^n$.  
\bigskip

 \begin{figure}[!htbp]
\begin{center}
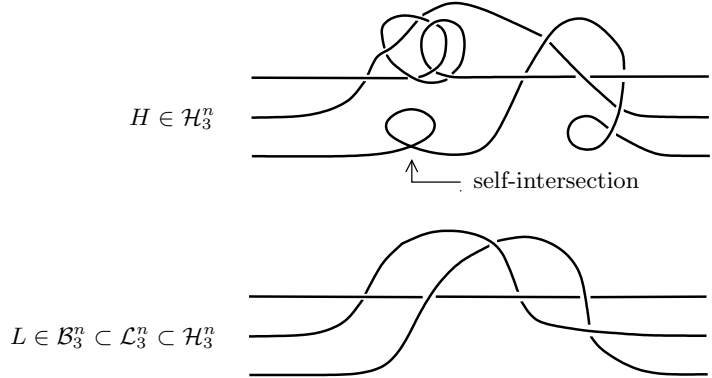
\caption{Examples of links. The top picture is a homotopy link, but not a link (and hence not a braid) because of the self-intersection in the bottom strand.}
\label{F:LinkExamples}
\end{center}
\end{figure}

When we say ``link", we will mean an embedded link.  Otherwise we will say ``homotopy link" or ``braid".  As with knots, the adjective ``long" will be dropped.  We will denote components (i.e.~strands) of an embedded link by $L=(K_1,K_2, ..., K_m)$.
 \bigskip
 
 As before, an \emph{isotopy} is a  homotopy in the space of links or braids, and \emph{link homotopy} is a path in the space of homotopy links.
\bigskip

As in the case of the space of knots, all of these link spaces are smooth infinite-dimensional paracompact manifolds so we can consider their deRham cohomology. 
\bigskip

Finite type invariants of these link spaces can be defined the same way as for knots (see \refS{FiniteType}).  Namely, we consider self-intersections which, in the case of links, come from a single strand or two different strands (i.e.~in the left picture of Figure \ref{F:SkeinRelation}, there are no conditions on the two strands making up the singularity).  For homotopy links, we only take intersections that come from different strands.  For braids, this condition is automatic since a braid cannot ``turn back" to intersect itself. Then a finite type $k$ invariant is defined as an invariant that vanishes on $(k+1)$ self-intersections.  We will denote finite type $k$ invariants of links, homotopy links, and braids by $\mathcal{LV}_k$, $\mathcal{HV}_k$, and $\mathcal{BV}_k$, respectively.
\bigskip

As for knots, the question of separation of links by finite type invariants is still open, but it is known that these invariants separate   homotopy links \cite{HabLin-Classif} and braids \cite{BN:HoLink, Kohno:LoopsFiniteType}.
\bigskip

We now revisit \refS{HigherKnots} and show how \refT{Italians} generalizes to links. Namely, recall the cochain map
$$
I_{\K^n}\colon \mathcal{D}^n \longrightarrow \Omega^*(\K^n).
$$
The first order of business is to generalize the diagram complex $\mathcal{D}^n$ to a complex $\mathcal{LD}_m^n$ (which we will use for both links and braids) and a subcomplex $\mathcal{HD}_m^n$ (which we will use for homotopy links).
This generalization is simple:  $\mathcal{LD}_m^n$ is defined the same way as $\mathcal{D}^n$ except there are now $m$ segments, as for example in Figure \ref{F:DiagramExample}.

 \begin{figure}[!htbp]
\begin{center}
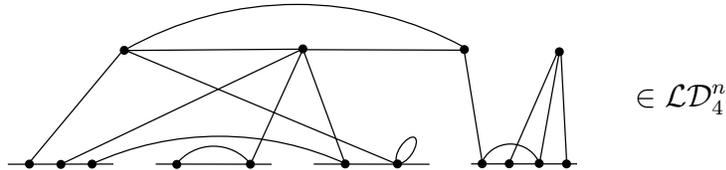
\caption{An example of a diagram for links (without decorations).}
\label{F:DiagramExample}
\end{center}
\end{figure}

All the definitions from \refS{HigherKnots} carry over in exactly the same way and we will not reproduce the details here, especially since they are spelled out in \cite[Section 3]{KMV:FTHoLinks}.  In particular, depending on the parity of $n$, the diagrams have to be appropriately decorated. The differential is again given by contracting arcs and edges.
\bigskip

$\mathcal{HD}_m^n$ is defined by taking diagrams 
\begin{itemize}
\item with no loops, and 
\item requiring that, if there exists a path between distinct vertices on a given segment, then it must pass through a vertex on another segment.  
\end{itemize}
It is a simple combinatorial exercise to show that $\mathcal{HD}_m^n$ is a subcomplex of $\mathcal{LD}_m^n$ \cite[Proposition 3.24]{KMV:FTHoLinks}.
\bigskip

As expected, in degree zero, complexes $\Lk_m^n$ and $\HLk_m^n$ are still defined by imposing the STU and IHX relations, and an extra relation in the case of $\HLk_m^n$ that diagrams cannot contain closed paths of edges.  In particular, the spaces of weight systems of degree $k$ for these link spaces, which we will denote by $\LW_k$ and $\HLW_k$, consist of functionals vanishing on these relations (with automorphism factors); see  \cite[Definition 3.35]{KMV:FTHoLinks} for details.
\bigskip

As it turns out, the integration is not as easily generalized.  The problem is that, if we want to produce cohomology classes on $\HLk_m^n$, then the evaluation map from \eqref{E:pqpullback} will need to take values in $\HLk_m^n$, but points in this space are not even immersions, let alone embeddings.  Hence the target of the evaluation map would not be a configuration space but rather some kind of a ``partial configuration space" where some points are allowed to pass through each other (this is actually a complement of a subspace arrangement, a familiar object from algebraic geometry).  But then the projection map would be a map of partial configuration spaces which is far from being a fibration (see \cite[Example 4.7]{KMV:FTHoLinks}).  This makes it unlikely that the pullback is a bundle over $\HLk_m^n$.
\bigskip

A way around this is to patch the integral together from pieces for which this difficulty does not occur.  This is achieved by breaking up a diagram $\Gamma\in \LD_m^n$ into its \emph{graft components} which are essentially the components one would see after the segments and segment vertices are removed.  The second condition defining the subcomplex $\HLD_m^n$ guarantees that there will be no more than one segment vertex on each segment of each graft, and this turns out to remove the issue of the projection not being a fibration.  Since it would take us too far afield to define the graft-based pullback bundle precisely, we will refer the reader to \cite[Definition 4.16]{KMV:FTHoLinks} for details.  Suffice it to say here that the construction essentially takes into account both the vertices and edges of $\Gamma$ rather than just vertices when constructing the pullback bundle (see Remark \ref{R:DisconnectedStratum}). This procedure is indeed a  refinement of the original definition of configuration space integrals since it produces the same form on $\Lk_m^n$ as the original definition \cite[Proposition 4.24]{KMV:FTHoLinks}.  The only difference, therefore, is that the graft definition makes it possible to restrict the integration from the complex $\LD_m^n$ to the subcomplex $\HLD_m^n$ and produce forms on $\HLk_m^n$.
\bigskip

We then have a generalization of \refT{Italians}.

\begin{thm1}[\cite{KMV:FTHoLinks}, Section 4.5]
There are integration maps $I_{\mathcal{L}}$ and $I_{\mathcal{H}}$ given by configuration space interals and a commutative diagram
$$
\xymatrix{
\mathcal{HD}_m^n \ar[r]^-{I_{\HLk}} \ar@{^{(}->}[d]  & \Omega^*(\HLk_m^n) \ar[d] \\
\mathcal{LD}_m^n \ar[r]^-{I_{\Lk}} &  \Omega^*(\Lk_m^n)
}
$$
Here $I_{\mathcal{L}}$ is a cochain map for $n>3$ and $I_{\mathcal{H}}$ is a cochain map for $n\geq 3$.
\end{thm1}

\begin{proof}
The proof goes exactly the same way as in the case of knots.  The only difference is that $I_{\mathcal{H}}$ is also a cochain map for $n=3$.  The reason for this is that the anomalous face is not an issue for homotopy links.  Namely, the anomaly can only arise when all points on and off the link collide.  But since strands of the link are always disjoint, this is only possible when all the configuration points on the link are in fact on a single strand.  In other words, the corresponding diagram $\Gamma$ is concentrated on a single strand.  Such a diagram does not occur in $\mathcal{HD}_m^n$.  (The integral in this case in effect produces a form on the space of knots, so that the anomaly can be thought of as a purely knotting, rather than linking, phenomenon.)
\end{proof}

\begin{rem1} As in the case of knots, there is an algebra structure on $\mathcal{LD}_m^n$ given by the shuffle product that is compatible with the wedge product of forms \cite[Section 3.3.2]{KMV:FTHoLinks}.  It thus turns out that the maps $I_{\Lk}$ and $I_{\HLk}$ are maps of algebras \cite[Proposition 4.29]{KMV:FTHoLinks}.
 \end{rem1}

For $n=3$, we also have a generalization of \refT{AllFiniteType}.

\begin{thm1}[\cite{KMV:FTHoLinks}, Theorems 5.6 and 5.8]\label{T:AllFiniteTypeLinks}
Configuration space integral maps $I_{\mathcal{L}}$ and $I_{\mathcal{H}}$ induce isomorphisms
\begin{align*}
I_{\mathcal{L}}^0\colon \LW_k  & \stackrel{\cong}{\longrightarrow} \mathcal{LV}_k/\mathcal{LV}_{k-1}\subset \Ho^0(\mathcal{L}_m^3) \\
I_{\mathcal{H}}^0\colon \HLW_k  & \stackrel{\cong}{\longrightarrow} \mathcal{LV}_k/\mathcal{LV}_{k-1}\subset \Ho^0(\mathcal{H}_m^3) 
\end{align*}
\end{thm1}

The isomorphisms are given exactly as in \refT{AllFiniteType}. In particular, the anomalous correction has to be introduced for the case of links.
\bigskip

Lastly, we mention an interesting connection to a class of classical homotopy link invariants called \emph{Milnor invariants} \cite{Milnor-Mu}.  In brief, these invariants live in the lower central series of the link group and essentially  measure how far a ``longitude" of the link survives in the lower central series. It is known that Milnor invariants of long homotopy links are finite type invariants (and it is important that these are long, rather than closed homotopy links) \cite{BN:HoLink, Lin:FTHoLink}.   \refT{AllFiniteTypeLinks} thus immediately gives us

\begin{cor1}
The map $I_{\mathcal{H}}$ provides configuration space integral expressions for Milnor invariants of $\HLk_m^3$.
\end{cor1}

For more details about this corollary, see \cite[Section 5.4]{KMV:FTHoLinks}.
\bigskip
 
Even though we made no explicit mention of braids in the above theorems, everything goes through the same way for this space as well.  The complex is still $\LD_m^n$ but the evaluation now take place on braids, i.e.~elements of  $\Br_m^n$.  The integration $I_{\Br}$ would thus produce forms on $\Br_m^n$ and all finite type invariants of $\Br_m^3$.  However, this is not very satisfying since we do not yet have a good subcomplex of $\LD_m^n$ or a modification in the integration procedure that would take into account the definition of braids.  For example, since $\Br_m^n\simeq \Omega\Conf(m, \R^{n-1})$, braids can be thought of as ``flowing" at the same rate, and the integration hence might be defined so that it only takes place in ``vertical slices" of the braid.   In particular, one should be able to connect configuration space integrals for braids to Kohno's work on braids and Chen integrals \cite{Kohno:LoopsFiniteType}.


\subsection{Manifold calculus of functors and finite type invariants}\label{S:Calculus}


Configuration space integrals connect in unexpected ways to the theory of \emph{manifold calculus of functors} \cite{W:EI1, GW:EI2}.  Before we can state the results, we provide some basic background, but we will assume the reader is familiar with the language of categories and functors.  For further details on  manifold calculus of functors, the reader might find \cite{M:MfldCalc} useful.
\bigskip

For $M$ a smooth manifold, 
Let $\Top$ be the category of topological spaces and let 
$$
\mathcal{O}(M)=\text{category of open subsets of $M$ with inclusions as morphisms}.
$$
Manifold calculus studies functors
$$
F\colon \mathcal{O}(M)^{op} \longrightarrow \Top
$$
satisfying the conditions:
\begin{enumerate}
\item $F$ takes isotopy equivalences to homotopy equivalences, and;
\item For any sequence of open sets $U_0\subset U_1\subset\cdots$, the canonical map $F(\cup_i U_i)\to\holim_i F(U_i)$ is a homotopy equivalence (here ``$\holim$" stands for the homotopy limit).
\end{enumerate}
The target category is not limited to topological spaces, but for concreteness and for our purposes we will stick to that case.
\bigskip

One such functor is the space of embeddings $\Emb(-,N)$, where $N$ is a smooth manifold, since, given an inclusion
 $$
O_1\hookrightarrow O_2$$ of open subsets of $M$, there is a restriction 
$$
\Emb(O_2,N)\longrightarrow \Emb(O_1,N).  
$$
In particular, we can specialize to the space of knots $\K^n$, $n\geq 3$, and see what manifold calculus has to say about it.
\bigskip

For any functor $F\colon \mathcal{O}(M)^{op} \to \Top$, the theory produces a ``Taylor tower" of approximating functors/fibrations
$$
F(-)\longrightarrow\big(
 T_{0}F(-)\leftarrow \cdots\leftarrow T_{k}F(-) \leftarrow \cdots \leftarrow T_{\infty}F(-)
\big)
$$
where $T_{\infty}F(-)$ is the inverse limit of the tower.

\begin{thm1}[\cite{W:HomEmb}]\label{T:HomologyConvergence}
For $F=\Emb(-,N)$ and for $2\dim(M)+2\leq \dim(N)$, the Taylor tower converges on (co)homology (for any coefficients), i.e.
$$\Ho_*(\Emb(-,N))\cong\Ho_*(T_{\infty}\Emb(-,N)).$$ 
In particular, evaluating at $M$ gives 
$$\Ho_*(\Emb(M,N))\cong\Ho_*(T_{\infty}\Emb(M,N)).$$ 
\end{thm1}
 
 \begin{rem1}
 For $dim(M)+3\leq dim(N)$, the same convergence result is true on homotopy groups \cite{GK}.
 \end{rem1}

Note that when $M$ is 1-di\-men\-sio\-nal, $N$ has to be at least 4-di\-men\-sio\-nal to guarantee convergence. Hence this says nothing about $\K^3$.  Nevertheless, the tower can still be constructed in this case and it turns out to contain a lot of information.      
\bigskip

To construct $T_k\K^n$, $n\geq 3$, let $I_1, ..., I_{k+1}$ be disjoint closed subintervals of $\R$ and let 
$$
\emptyset\neq S\subseteq\{1,..., k+1\}.
$$ 
Then let
$$
\K^n_S=\Emb_c(\R\setminus \bigcup_{i\in S}I_i, \ \R^n),
$$
where $\Emb_c$ as usual stands for the space of ``compactly supported" embeddings, namely those that are fixed outside some compact set such as $I$.
\bigskip

Thus $\K^n_S$ is a space of ``punctured knots"; an example is given in Figure \ref{F:PuncturedKnot}.

\begin{figure}[height=5mm]
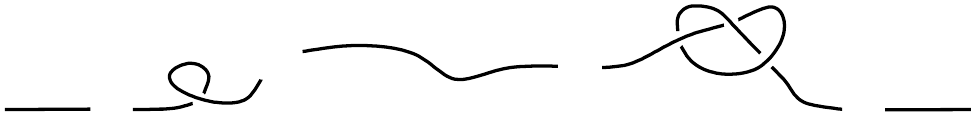
\caption{An element of $\K_{\{1,2,3,4\}}^n$.}
\label{F:PuncturedKnot}
\end{figure}

These spaces are not very interesting on their own, and are connected even for $n=3$. However, we have restriction maps $\K^n_S\to\K^n_{S\cup\{i\}}$ given by ``punching another hole", namely restricting an embedding of $\R$ with some intervals taken out to an embedding of $\R$ with one more interval taken out.  These spaces and maps then form a diagram of knots with holes (such a diagram is sometimes called a  \emph{punctured cube}).  

\begin{exam1}\label{Ex:KnotsCubeExample}
When $k=2$, we get
$$
\xymatrix@=10pt{
         &           &   \K^n_{\{1\}} \ar'[d][dd]           \ar[dr]  &                  \\
        &  \K^n_{\{2\}} \ar[rr] \ar[dd]  &             & \K^n_{\{1,2\}}
         \ar[dd] \\
\K^n_{\{3\}} \ar'[r][rr] \ar[dr] &        &   \K^n_{\{1,3\}}
\ar[dr] &                   \\
        &   \K^n_{\{2,3\}} \ar[rr]      &                    &
\K^n_{\{1,2,3\}}
}
$$ 
\end{exam1}

\begin{defi1}  The $k$th stage of the Taylor tower for $\K^n$, $n\geq 3$, is the homotopy limit of this punctured cube, i.e.
$$
T_k\K^n=\underset{\emptyset\neq S\subseteq\{1, .., k+1\}}{\holim} \K^n_S.
$$
\end{defi1}

For the reader not familiar with homotopy limits, it is actually not hard to see what this homotopy limit is:  For example, the punctured cubical diagram from Example \ref{Ex:KnotsCubeExample} can be redrawn as
$$
\xymatrix@=10pt{
                                     &                                      &     \K^n_{\{1\}}  \ar[dl] \ar[dr]  &                                &                                 \\
                                     &      \K^n_{\{1,3\}} \ar[r]           &    \K^n_{\{1,2,3\}}                  &   \K^n_{\{1,2\}} \ar[l]         &                                 \\
  \K^n_{\{3\}} \ar[ur] \ar[rr]    &                                       &     \K^n_{\{2,3\}} \ar[u]         &                                &    \K^n_{\{2\}}\ar[ll]\ar[ul]  
  }
$$

Then a point in $T_2\K^n$ is 
\begin{itemize}
\item A point in each $\K^n_{\{i\}}$ (once-punctured knot);
\item A path in each $\K^n_{\{i,j\}}$ (isotopy of a twice-punctured knot) ;
\item A two-parameter path in $\K^n_{\{1,2,3\}}$ (two-parameter isotopy of a thrice-punctured knot); and
\item Everything is compatible with the restriction maps.  Namely, a path in each $\K^n_{\{i,j\}}$ joins the restrictions of the elements of $\K^n_{\{i\}}$ and $\K^n_{\{j\}}$ to $\K^n_{\{i,j\}}$, and the two-parameter path in $\K^n_{\{1,2,3\}}$ is really a map of a $2$-simplex into $\K^n_{\{1,2,3\}}$ which, on its edges, is the restriction of the paths in $\K^n_{\{i,j\}}$ to $\K^n_{\{1,2,3\}}$.
\end{itemize}
The pattern for $T_k\K^n$, $k\neq 2$, should be clear.
\bigskip

There is a map
$$
\K^n\longrightarrow T_k\K^n
$$
given by punching holes in the knot (the isotopies in the homotopy limit are thus constant).

\begin{rem1}  It is not hard to see that, for $k\geq 3$, $\K^n$ is the actual pullback (limit) of the subcubical diagram.  
So the strategy here is to replace the limit, which is what we really care about, by the homotopy limit, which is hopefully easier to understand.
\end{rem1}

There is also a map, for all $k\geq 1$,
$$
T_k\K^n\longrightarrow T_{k-1}\K^n,
$$
since the diagram defining $T_{k-1}\K^n$ is a subdiagram of the one defining $T_k\K^n$ (this map is a fibration; this is a general fact about homotopy limits).
\bigskip

Putting these maps and spaces together, we get the Taylor tower for $\K^n$, $n\geq 3$:
$$
\K^n\longrightarrow\big(
 T_{0}\K^n\leftarrow \cdots\leftarrow T_{k}\K^n \leftarrow \cdots \leftarrow T_{\infty}\K^n
\big).
$$
By \refT{HomologyConvergence}, this tower converges on (co)homology for $n\geq 4$.
\bigskip

There is a variant of this Taylor tower, the so-called ``algebraic Taylor tower", which is a tower of cochain complexs obtained by applying cochains to each space of punctured knots and then letting $T_k^*(\K^3)$ be the homotopy colimit of the resulting diagram of cochain complexes.
\bigskip

Recall the map $I_{\K^3}^0$ from \refT{AllFiniteType}.  We then have the following theorem, which essentially says that the algebraic Taylor tower classifies finite type invariants. 
\begin{thm1}[\cite{V:FTK}, Theorem 1.2]\label{T:MainThesisTheorem}  The map $I_{\K^3}^0$ factors through the algebraic Taylor tower for $\K^3$.  Furthermore, we have isomorphisms
$$
\xymatrix{
\mathcal{TW}_k\ar[rr]^{I_{\K^3}^0}_{\cong}\ar[dr]^{\cong}  &  &  \calV_k/\calV_{k-1} \\
&  \Ho^0(T_{2k}^*\K^3) \ar[ur]^{\cong} & 
}
$$
(and $\Ho^0(T_{2k}^*\K^3)\cong\Ho^0(T_{2k+1}^*\K^3)$ so all stages are accounted for).
\end{thm1}

The main ingredient in this proof is the extension of configuration space integrals to the stages $T_{2k}\K^3$ of the space of long knots \cite[Theorem 4.5]{V:IT}.
The idea of this extension is this:  As a configuration point moves around a punctured knot (this corresponds to a point moving on the knot in the usual construction) and approaches a hole, it is made to ``jump", via a path in the homotopy limit (this is achieved by an appropriate partition of unity), to another punctured knot which has that hole filled in, thus preventing the evaluation map from being undefined.
\bigskip

\refT{MainThesisTheorem} places finite type theory into a more homotopy-theo\-re\-tic setting and the hope is that this might give a new strategy for proving the separation conjecture.
\bigskip

Several generalizations of \refT{MainThesisTheorem} should be possible.  For example, it should also be possible to extend the entire chain map from \refT{Italians} to the Taylor tower. It should also be possible to show that the Taylor \emph{multitowers} for $\Lk_m^n$, $\HLk_m^n$, and $\Br_m^n$ supplied by the \emph{multivariable manifold calculus of functors} \cite{MV:Multi} also admit factorizations of the integration maps $I_{\Lk}$, $I_{\HLk}$, and $I_{\Br}$, as well as classify finite type invariants of these spaces.  
\bigskip

Lastly, it seems likely that finite type invariants are all one finds in the ordinary Taylor tower (and not just its algebraic version).  Some progress toward this goal can be found in \cite{BCSS}.


\subsection{Formality of the little balls operad}\label{S:Formality}


There is another striking connection between the Taylor tower for knots and configuration space integrals of a slightly different flavor.  To explain, we will first modify the space of knots slightly and instead of $\K^n$, $n>3$, use the space
$$
\overline{\K^n}=\hofiber (\Emb_c(\R, \R^n)\hookrightarrow \Imm_c(\R, \R^n)),
$$
where $\hofiber$ stands for the homotopy fiber.
This is the space of ``embeddings modulo immersions" and a point in it is a long knot along with a path, i.e.~a regular homotopy, to the long unknot (since this is a natural basepoint in the space of immersions) through compactly supported immersions.  This space is easier to work with, but it is not very different from $\K^n$:  the above inclusion is nullhomotopic \cite[Proposition 5.17]{S:OKS} so that we have a homotopy equivalence
$$
\overline{\K^n}\simeq \K^n\times \Omega^2 S^{n-1}.
$$
The difference between long knots and its version modulo immersions is thus well-undersrtood, especially rationally.
\bigskip

Now let $\mathcal{B}_n=\{\mathcal B_n(p)\}_{p\geq 0}$ be the little $n$-discs operad (i.e.~little balls in $\R^n$), where $\mathcal B_n(p)$ is the space of configurations of $p$ closed $n$-discs with disjoint interiors contained in the unit disk of $\R^n$.  The little $n$-discs operad is an important object in homotopy theory, and a good introduction  for the reader who is not familiar with it, or with operads in general, is \cite{MMS:Operads}.  Taking the chains and the homology of $\mathcal{B}_n$ gives two operads of chain complexes, $\operatorname{C}_*(\mathcal{B}_n; \R)$ and $\operatorname{H}_*(\mathcal{B}_n; \R)$ (where the latter is a chain complex with zero differential).  Then we have the following \emph{formality} theorem.

\begin{thm1}[Kontsevich \cite{K:OMDQ}; Tamarkin for $n=2$ \cite{Tamarkin:Formality}]\label{T:Formality}
For $n\geq 2$, there exists a chain of weak equivalences of operads of chain complexes
$$
\operatorname{C}_*(\mathcal{B}_n; \R)\stackrel{\simeq}{\longleftarrow}\widetilde{\calD^n}
\stackrel{\simeq}{\longrightarrow} \operatorname{H}_*(\mathcal{B}_n; \R),
$$
where $\widetilde{\calD^n}$ is a certain diagram complex.
In other words, $\mathcal{B}_n$ is (stably) formal over $\R$.
\end{thm1}

For details about the proof of \refT{Formality}, the reader should consult \cite{LV}.  The reason this theorem is relevant here is that $\widetilde{\calD^n}$ is a diagram complex (it is in fact a commutative differential graded algebra cooperad) that is essentially the complex $\calD^n$ we encountered in \refS{HigherKnots}.  The main difference is that loops are not allowed in $\widetilde{\calD^n}$.  In addition, the map 
$$
\widetilde{\calD^n}\longrightarrow \operatorname{C}_*(\mathcal{B}_n; \R)
$$
is given by configuration space integrals and is the same as the map $I_{\K^n}$ we saw before, with one important difference that the bundle we integrate over is different.  To explain briefly, $\mathcal{B}_n$ can be regarded as a collection of configuration spaces $\Conf[p,\R^n]$.  For a diagram $\Gamma\in\widetilde{\calD^n}$ with $p$ segment vertices and $q$ free vertices, consider the projection
$$
\pi\colon \Conf[p+q, \R^n]\longrightarrow \Conf[p, \R^n].
$$
This is a bundle in a suitable sense; see \cite[Section 5.9]{LV}.  The edges of $\Gamma$ again determine some Gauss maps $\Conf[p+q, \R^n]\to S^{n-1}$, so that the product of volume forms can be pulled back to $\Conf[p+q, \R^n]$ and then pushed forward to $\Conf[p, \R^n]$.  This part is completely analogous to what we have seen in \refS{HigherKnots}.  The bulk of the proof of \refT{Formality} then consist of showing that the map $\widetilde{\calD^n}\to\Omega^*(\Conf[p,\R^n])$ is an equivalence (we are liberally passing between cooperad $\widetilde{\calD^n}$ and its dual, as well as chains and cochains on configuration spaces).  This again comes down to vanishing arguments, which are the same as in \refT{Italians}.  
\bigskip

The map 
$$
\widetilde{\calD^n}\longrightarrow \operatorname{H}_*(\mathcal{B}_n; \R)
$$
is easy, with some obvious diagrams sent to the generators of the (co)homology of configuration spaces and the rest to zero.
\bigskip

\refT{Formality} has been used in a variety of situations, such as McClure-Smith's proof of the Deligne Conjecture and Tamarkin's proof of Kontsevich's deformation quantization theorem.  For us, the importance is in that it gives information about the rational homology of $\overline{\K^n}$, $n>3$.  
\bigskip

To explain, first observe that the construction of the Taylor tower for $\K^n$ from \refS{Calculus} can be carried out in exactly the same way for $\overline{\K^n}$.  Then, by retracting the arcs of a punctured knot, we get
$$
\K^n_S\simeq\Conf(|S|-1, \R^n).
$$
If we had used $\K^n$, we would also have copies of spheres keeping track of tangential data.  In the $\overline{\K^n}$ version, they are not present since the space of immersions, which carries this data, has been removed. 
\bigskip

The restriction maps ``add a point", as in Figure \ref{F:PuncturedKnotRestrict}.

\begin{figure}[!htbp]
\begin{center}
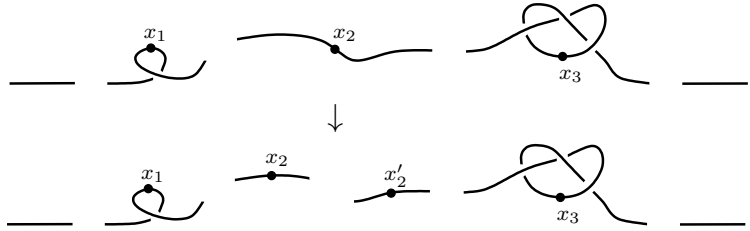
\caption{Restriction of punctured knots.}
\label{F:PuncturedKnotRestrict}
\end{center}
\end{figure} 

This structure yields a homology spectral sequence that can be associated to the Taylor tower for $\overline{\K^n}$, $n\geq 3$.  It starts with
$$
E^{1}_{-p,q}=\Ho_q(\Conf(p,\R^n)),
$$
and, for $n\geq 4$, converges to 
$
\Ho_*(T_{\infty}\overline{\K^n})$.  By \refT{HomologyConvergence}, this spectral sequence hence converges to $\Ho_*(\overline{\K^n})$.

\begin{rem1}
This is the \emph{Bousfield-Kan} spectral sequence that can be associated to any \emph{cosimplicial space}, and in particular to the cosimplicial space defined by Sinha \cite{S:TSK} that models the Taylor tower for $\overline{\K^n}$.  The spaces in this cosimplicial model are slight modifications of the Fulton-MacPherson compactification of $\Conf(p,\R^n)$ and the maps ``double" and ``forget" points. In particular, the doubling maps are motivated by the situation from Figure \ref{F:PuncturedKnotRestrict}. In addition, this turns out to be the same spectral sequence (up to regrading) as the one defined by Vassiliev \cite{Vass:Cohom} which motivated  the original definition of finite type knot invariants.
\end{rem1}
 
\begin{thm1}
The homology spectral sequence described above collapses rationally at the $E^2$ page for $n\geq 3$.
\end{thm1}

This theorem was proved for $n\geq 4$ in \cite{LTV:Vass}, for $n= 3$ on the diagonal in \cite{K:Fey}, and for $n=3$ everywhere in \cite{Moriya:SSCollapse} and \cite{PAST:SSCollapse}.  \refT{MainThesisTheorem} can also be interpreted as the collapse of this spectral sequence on the diagonal for $n=3$.
 
 \begin{proof}[Idea of proof]
The vertical differential in the spectral sequence is the ordinary internal differential on the cochain complexes of configuration spaces (the vertical one has to do with doubling configuration points, and this has to do with Figure \ref{F:PuncturedKnotRestrict}).  By \refT{Formality}, this differential can be replaced by the zero differential, and hence the spectral sequence collapses.  Some more sophisticated model category theory techniques are required for the case $n=3$.
 \end{proof}

\begin{rem1}
Collapse is also true for the \emph{homotopy} spectral sequence for $n\geq 4$ \cite{ALTV}.
\end{rem1}

So for $n\geq 4$, the homology of the $E^2$ page is the homology of $\overline{\K^n}$.
A more precise way to say this is
\begin{thm1}
For $n\geq 4$, $\Ho_*(\overline{\K^n}; \Q)$ is the Hochschild homology of the Poisson operad of degree $n-1$, which is the  operad obtained by taking the homology of the little $n$-cubes operad.
\end{thm1}

For more details on the Poisson operad of degree $n-1$, see \cite[Section 1]{T:Homology}.  Briefly, this is the operad encoding Poisson algebras, i.e.~graded commutative algebras with Lie bracket of degree $n-1$ which is a graded derivation with respect to the multiplication.  One can define a differential (the Gerstenhaber bracket), and the homology of the resulting complex is the Hochschild homology of the Poisson operad.
\bigskip

In summary, $\Ho_*(\overline{\K^n}; \Q)$ is built out of $\Ho_*(\Conf(p,\R^n); \Q)$, $p\geq 0$, which is understood.  In fact, this homology can be represented combinatorially with graph complexes of chord diagrams.  This therefore gives  a nice combinatorial description of $\Ho_*(\K^n; \Q)$, $n\geq 4$.  The case $n=3$ is not yet well understood, and the implications of the collapse are yet to be studied.  The main impediment is that, even though the spectral sequence collapses, it is not clear what the spectral sequence converges to.
\bigskip

It would be nice to rework the results described in this section for spaces of links.  For ordinary embedded links and braids, things should work the same, but homotopy links are more challenging since we do not yet have any sort of a convergence result for the Taylor (multi)tower for this space.

\bigskip
{\bf \centerline{Acknowledgements}}
The author would like to thank Sadok Kallel for the invitation to write this article and the referee for a careful reading.


\bigskip
\hfill\
{\footnotesize
\parbox{5cm}{Ismar Voli\'c\\
{\it Department of Mathematics},\\
Wellesley College,\\
106 Central Street,\\
Wellesley, MA 02481,\\
{\sf ivolic@wellesley.edu}}\
{\hfill}\
}

\end{document}